\documentclass[12pt]{amsart}
\usepackage{amsfonts, amssymb, latexsym, hyperref, xcolor, dynkin-diagrams, tikz,tikz-cd}
\usepackage{ytableau}

\newtheorem{theorem}{Theorem}[section]
\newtheorem{proposition}[theorem]{Proposition}
\newtheorem{lemma}[theorem]{Lemma}

\newtheorem{Claim}[theorem]{Claim}
\newtheorem{claim}[theorem]{Claim}
\newtheorem*{claim*}{Claim}
\newtheorem{corollary}[theorem]{Corollary}
\newtheorem{Main Conjecture}[theorem]{Main Conjecture}
\newtheorem{conjecture}[theorem]{Conjecture}
\newtheorem{problem}[theorem]{Problem}
\theoremstyle{definition}
\newtheorem{definition}[theorem]{Definition}
\theoremstyle{remark}

\newtheorem{example}[theorem]{Example}
\newtheorem{Remark}[theorem]{Remark}
\newtheorem{remark}[theorem]{Remark}
\theoremstyle{plain}

\newcommand{\lineb}[1]{\mathcal{L}_{#1}}

\newcommand{\cellsize}{12}
\newlength{\cellsz} 
\newsavebox{\cell}
\sbox{\cell}{\begin{picture}(\cellsize,\cellsize)
\put(0,0){\line(1,0){\cellsize}}
\put(0,0){\line(0,1){\cellsize}}
\put(\cellsize,0){\line(0,1){\cellsize}}
\put(0,\cellsize){\line(1,0){\cellsize}}
\end{picture}}
\newcommand\cellify[1]{\def\thearg{#1}\def\nothing{}%
\ifx\thearg\nothing
\vrule width0pt height\cellsz depth0pt\else
\hbox to 0pt{\usebox{\cell} \hss}\fi%
\vbox to \cellsz{
\vss
\hbox to \cellsz{\hss$#1$\hss}
\vss}}
\newcommand\tableau[1]{\vtop{\let\\\cr
\baselineskip -16000pt \lineskiplimit 16000pt \lineskip 0pt
\ialign{&\cellify{##}\cr#1\crcr}}}

\newcommand{\excise}[1]{}

\begin{document}
\pagestyle{plain}
\title{Coxeter combinatorics and spherical Schubert geometry}
\author{Reuven Hodges}
\address{Dept.~of Mathematics, University of Illinois at Urbana-Champaign, Urbana, IL 61801\newline
Current address: Department of Mathematics, UC San Diego, La Jolla, CA 92093}
\email{rhodges@ucsd.edu} 
\author{Alexander Yong}
\address{Dept.~of Mathematics, University of Illinois at Urbana-Champaign, Urbana, IL 61801} 
\email{ayong@illinois.edu}
\date{\today}
\maketitle

\begin{abstract}
For a finite Coxeter system and a subset of its diagram nodes, we define \emph{spherical elements} (a generalization of 
\emph{Coxeter elements}).
Conjecturally, for Weyl groups, spherical elements index Schubert varieties in a flag manifold $G/B$ that are spherical for the action of a Levi subgroup. We evidence the conjecture, employing the 
combinatorics of \emph{Demazure modules}, and work of
R.~Avdeev--A.~Petukhov, M.~Can--R.~Hodges, R.~Hodges--V.~Lakshmibai, P.~Karuppuchamy, P.~Magyar--J.~Weyman--A.~Zelevinsky, N.~Perrin, J.~Stembridge, and B.~Tenner. In type $A$, we establish connections  with the \emph{key polynomials} of A.~Lascoux--M.-P.~Sch\"utzenberger, multiplicity-freeness, and \emph{split-symmetry} in algebraic combinatorics. Thereby, we invoke theorems  of A.~Kohnert, V.~Reiner--M.~Shimozono, and C.~Ross--A.~Yong. 
\end{abstract}

\section{Introduction}\label{sec:1}

\subsection{Main definition}
Let $(W,S)$ be a finite Coxeter system, where $S=\{s_1,\ldots,s_r\}$ are minimal generators of the Coxeter group $W$. Biject $[r]:=\{1,2,\ldots,r\}$ with the nodes of the Coxeter 
diagram ${\mathcal G}$. To each $I\in 2^{[r]}$, let ${\mathcal G}_I$ be the induced subdiagram of ${\mathcal G}$.  Suppose 
\begin{equation}
\label{eqn:thedecompabc}
{\mathcal G}_I=\bigcup_{z=1}^m {\mathcal C}^{(z)}
\end{equation}
is the decomposition into connected components. Let $w_0^{(z)}$ be the long element of the parabolic subgroup
$W_{I^{(z)}}$ generated by $I^{(z)}=\{s_j: j\in {\mathcal C}^{(z)}\}$. 

Every $w\in W$ has a \emph{reduced expression} $w=s_{i_1}\cdots s_{i_k}$ where $k=\ell(w)$ is the \emph{Coxeter 
length} of $w$. Let ${\text{Red}}(w):={\text{Red}}_{(W,S)}(w)$ be the set of these expressions.  The \emph{left descents} of $w$ are 
\[J(w)=\{j \in [r]: \ell(s_j w)<\ell(w)\}.\]

\begin{definition}[$I$-spherical elements]\label{def:main}
Let $w\in W$ and fix $I\subseteq J(w)$. Then $w$ is \emph{$I$-spherical} if there exists 
$R=s_{i_1}\cdots s_{i_{\ell(w)}}\in {\text{Red}}(w)$ such that:
\begin{itemize}
\item[(S.1)] $\#\{t: i_t= j\}\leq 1$ for all $j \in [r]- I$, and
\item[(S.2)] $\#\{t: i_t \in {\mathcal C}^{(z)}\} \leq \ell(w_0^{(z)}) +\#\text{vertices}({\mathcal C}^{(z)})$ for $1\leq z\leq m$.
\end{itemize}
Such an $R$ is called an \emph{$I$-witness}. Call $w$ \emph{maximally spherical} if it is $J(w)$-spherical. 
\end{definition}

\begin{example}[Coxeter elements]
A \emph{Coxeter element} $c$ of $W$ is the product of all $s_i$'s (in some order). Trivially,
$c$ is $I$-spherical for any $I\subseteq J(c)$.
\end{example}

\begin{example}
The $E_8$ Coxeter diagram is
\dynkin[label, edge length=0.5cm]E8. Let 
\[R=s_2 s_3 s_4 s_2 s_3 s_4 s_5 s_4 s_2 s_3 s_1 s_4 s_5 s_6 s_7 s_6 s_8 s_7 s_6 \in {\text{Red}}(w).\]
Then $J(w)=\{2,3,4,5,7,8\}$. If $I=J(w)$ then
${\mathcal C}^{(1)}=$\dynkin[labels={3,4,2,5}, edge length=0.5cm]D4 and  ${\mathcal C}^{(2)}=$\dynkin[labels={7,8}, edge length=0.5cm]A2.

Here $W_{I^{(1)}}$ is the $D_4$ Coxeter group and
$w_0^{(1)}=s_3 s_2 s_4 s_3 s_2 s_4 s_5 s_4 s_3 s_2 s_4 s_5$
with $\ell(w_0^{(1)})=12$. Also, $W_{I^{(2)}}$ is the $A_2$ Coxeter group ${\mathfrak S}_3$,
$w_0^{(2)}=s_7 s_8 s_7$ and $\ell(w_0^{(2)})=3$.

$R$ is not a $J(w)$-witness for $w$; it fails (S.1) as $s_6$ appears thrice.
However,
\begin{align*}
R & = s_2 s_3 s_4 s_2 s_3 s_4 s_5 s_4 s_2 s_3 s_1 s_4 s_5 {\color{blue} s_6 s_7 s_6} s_8 s_7 s_6\\
& \equiv s_2 s_3 s_4 s_2 s_3 s_4 s_5 s_4 s_2 s_3 s_1 s_4 s_5 s_7 s_6 {\color{blue} s_7 s_8 s_7} s_6\\
& \equiv s_2 s_3 s_4 s_2 s_3 s_4 s_5 s_4 s_2 s_3 s_1 s_4 s_5 s_7 {\color{blue} s_6 s_8} s_7 {\color{blue} s_8 s_6}\\
& \equiv s_2 s_3 s_4 s_2 s_3 s_4 s_5 s_4 s_2 s_3 s_1 s_4 s_5 s_7 s_8 {\color{blue} s_6 s_7 s_6} s_8\\
& \equiv s_2 s_3 s_4 s_2 s_3 s_4 s_5 s_4 s_2 s_3 s_1 s_4 s_5 s_7 s_8 s_7 s_6 s_7 s_8.
\end{align*}
The latter expression is a $J(w)$-witness.
\end{example}

\begin{example}[$B_2,B_3$]
For $B_2$, all elements are $J(w)$-spherical (Proposition~\ref{prop:dihedral}). For $B_3$, the diagram is
\dynkin[label, edge length=0.5cm]B3, and $\#W(B_3)=2^33!=48$. The $8$ non-$J(w)$-spherical elements are:
$s_3 s_2 s_3 s_1 s_2 s_3, \  s_2 s_3 s_2 s_1 s_2 s_3, \ s_3 s_2 s_3 s_2 s_1 s_2 s_3, \ s_3 s_2 s_3 s_1 s_2 s_3 s_2,\ 
 s_2 s_3 s_2 s_1 s_2 s_3 s_2,\ s_3 s_2 s_3 s_2 s_1 s_2 s_3 s_2$, \\
    $s_2 s_3 s_2 s_1 s_2, \  s_3 s_2 s_3 s_1 s_2 s_3 s_2 s_1$.
\end{example}

\begin{example}[$F_4$]
The $F_4$ diagram is \dynkin[label, edge length=0.5cm]F4. Of the $1152$ Weyl group elements, $290$ are
$J(w)$-spherical. An example is $w=s_4 s_3 s_4 s_2 s_3 s_4 s_2 s_3 s_2 s_1 s_2 s_3 s_4$ (here $J(w)=\{2,3,4\}$). A non-example
is $w'=s_2  s_1  s_4 s_3 s_2 s_1 s_3 s_2 s_4 s_3 s_2 s_1$ ($J(w')=\{2,4\}$);  here $\#\text{Red}(w')=29$.
\end{example}

This paper will concentrate mainly on type $A_{n-1}$ \dynkin[labels={1,2, \ ,n\!-\! 1}, edge length=0.5cm]A{}. $W(A_{n-1}) \cong {\mathfrak S}_n$, the symmetric group on $\{1,2,\ldots,n\}$. Each $s_i$ is identified with the transposition $(i \ i+1)$.

\begin{example}\label{ex:S123456} All $w\in {\mathfrak S}_n$ are $J(w)$-spherical, if $n\leq 4$. In ${\mathfrak S}_5$ the non-examples are
\[24531, 25314, 25341, 34512, 34521, 35412, 35421, 42531, 45123, 45213, 45231,\] 
\[45312, 52314, 52341, 53124, 53142, 53412, 53421, 54123, 54213, 54231.\]
There are $320$ non-examples in ${\mathfrak S}_6$, and $3450$ in ${\mathfrak S}_7$ (the latter
computed by J.~Hu). We suspect that, for $n$ large, nearly all $w\in {\mathfrak S}_n$ are non-examples (Conjecture~\ref{conj:vanishing}). Notice $24531^{-1}=51423$ is not on the list. Being maximally spherical is not
inverse invariant.
\end{example}

\begin{example}[$321$-avoiding permutations] $w\in {\mathfrak S}_n$ is \emph{$321$-avoiding} if there does not exist $i<j<k$ such that
$w(i)>w(j)>w(k)$. Such $w$ are \emph{fully commutative}, \emph{i.e.}, no expression in ${\text{Red}}(w)$
contains $s_i s_{i+1} s_i$ nor $s_{i+1} s_i s_{i+1}$. Any two elements of ${\text{Red}}(w)$ can be obtained
from one another by a sequence of \emph{commutation relations} $s_i s_j \equiv s_j s_i$ where $|i-j|\geq 2$ (see, \emph{e.g.},
\cite[Proposition~2.2.15]{Manivel}). Hence,
for any $I \in 2^{J(w)}$, the property of being an $I$-witness is independent of the choice of $s_{i_1}\cdots s_{i_{\ell(w)}}\in
{\text{Red}}(w)$.
\end{example}

\subsection{Spherical elements and Schubert geometry}
 Let $G$ be a connected complex reductive algebraic group. Fix a choice of maximal torus $T$ and Borel
subgroup $B$ in $G$ with root system $\Phi$ and decomposition into positive and negative roots 
$\Phi=\Phi^+\cup \Phi^-$. Let $\Delta$ be the base of the root system. The finite Coxeter group of interest
is the \emph{Weyl group} of $G$, namely $W\cong N(T)/T$. Let ${\text{rank}}_{ss}(G)$ be the semisimple rank of $G$. Then $W$ is generated by $r={\text{rank}}_{ss}(G)$ many simple reflections $S=\{s_{1},\ldots, s_r\}$, where $1,2,\ldots,r$ is some indexing of $\Delta$.

This paper builds on and extends 
earlier work of, \emph{e.g.}, P.~Magyar--J.~Weyman--A.~Zelevinsky \cite{MWZ}, J.~Stembridge \cite{Stem:Weyl}, 
P.~Karuppuchamy \cite{Karupp}, as well as work of the first author with
V.~Lakshmibai \cite{Hodges.Lakshmibai, Hodges.Lakshmibai:II} and with M.~Can \cite{Can.Hodges:2020}. It combines  study
of two topics of combinatorial algebraic geometry:

\begin{itemize}
\item[(A)] The \emph{generalized flag variety} is $G/B$. The \emph{Schubert varieties} are the $B$-orbit closures 
$X_w=\overline{BwB/B}$ where $w\in W$. Schubert varieties are well-studied in
algebraic combinatorics, representation theory and algebraic geometry; see, \emph{e.g.}, \cite{Fulton, BL00}.
\item[(B)] A variety $X$ is \emph{$H$-spherical} for the action of a complex reductive algebraic group $H$ 
if $X$ is normal and it contains a open dense orbit of a Borel subgroup of $H$. Spherical varieties generalize toric varieties. 
Classifying spherical varieties is of significant interest; see, \emph{e.g.}, \cite{Brion.Luna.Vust, Luna}, and the survey by N.~Perrin \cite{Perrin}. 
\end{itemize}

Foundational work from the 1980s, by C.~DeConcini--V.~Lakshmibai \cite{DL}, as well as
S.~Ramanan--A.~Ramanathan \cite{RR85}, established that every Schubert variety is normal. Thus to be within  (B)'s scope,
it remains to introduce a reductive group $H$ acting on $X_w$ ($H=B$ being invalid, as $B$ is not reductive).

We study a natural choice of $H$ acting on $X_w$. Recall, for any parabolic subgroup $P$ of $G$, the \emph{Levi decomposition} is
\begin{equation}
\label{eqn:levidecomposition}
P=  L\ltimes R_u(P)
\end{equation}
where $L$ is a \emph{Levi subgroup} of $P$ and $R_u(P)$ is $P$'s unipotent radical. 
For each $I\in 2^{[r]}$ there is a standard parabolic $P_I\supset B$; let $L_I$ be the associated
standard Levi from (\ref{eqn:levidecomposition}) that contains~$T$. With respect to the left action of $G$ on $G/B$,  
\begin{equation}
\label{eqn:stablizeractionparabolic}
P_{J(w)}={\rm stab}_{G}(X_w);
\end{equation}
see \cite[Lemma~8.2.3]{BL00}.  For any 
\[I\subseteq J(w), L_I\leq P_I\leq P_{J(w)}.\] 
Hence by (\ref{eqn:stablizeractionparabolic}) each of the \emph{reductive} groups $H=L_I$ acts on $X_w$. 

\begin{definition}\label{def:geomain}
Let $I\subseteq J(w)$. $X_w\subseteq G/B$ is \emph{$L_I$-spherical} if $X_w$ has an open dense orbit of a Borel subgroup of $L_I$
under left translations. $X_w$ is \emph{maximally spherical}
if it is $L_{J(w)}$-spherical.
\end{definition}

\begin{center}
Which Schubert varieties $X_w$ are spherical for the action of $L_I$?
\end{center}

\begin{conjecture}
\label{conj:main}
Let $I\subseteq J(w)$. $X_w$ is ${L}_I$-spherical if any only if $w$ is $I$-spherical.
\end{conjecture}

Condition (S.2) has the following Lie theoretic origin: if $G$ is semisimple and $B$ is a Borel subgroup, then
$\dim B=\ell(w_0)+{\text{rank}}(G)$. However, Conjecture~\ref{conj:main} predicts that being $L_I$-spherical only depends on the
Coxeter data. In particular, this suggests the sphericality classification is the same for 
$SO_{2n+1}/B$ vs.  $Sp_{2n}/B$.
 
To summarize earlier work, it seems nontrivial to certify sphericality of $X_w$, even in specific instances.  A certificate 
that $X_w$ is \emph{not} $I$-spherical is implicit in \cite{Perrin}. We expound upon it using research from algebraic combinatorics
(see Theorem~\ref{thm:fundamentalRelationship}). 

\begin{example}\label{exa:CanHodges}
M.~Can and the first author \cite[Theorems~6.2, 6.3]{Can.Hodges:2020} proved  that all Schubert varieties
in $SL_3/B$ and $SL_4/B$ are maximally spherical. This is consistent with Example~\ref{ex:S123456}. The methods of Section~\ref{sec:4} allow one to verify that the non-spherical cases shown in ${\mathfrak S}_5$ (and those alluded to in ${\mathfrak S}_6$) are 
indeed geometrically non-spherical.
\end{example}

\begin{example}[Toric Schubert varieties]\label{exa:toricJune17}
(S.1) was inspired by P.~Karuppuchamy's~\cite{Karupp}. In \emph{ibid.}, the author classified when 
$X_w$ is \emph{toric} with respect to $T$, that is, $X_w$ contains a dense orbit of $T$. Identically, this is classifying $L_{\emptyset}$-spherical $X_w$.  Indeed when $I=\emptyset$, (S.2) is a vacuous condition, and ``$X_w$ is toric $\iff$ (S.1)'' 
is precisely his classification. Earlier, B.~Tenner \cite{Tenner} proved (without reference to toric Schubert
geometry) that $w$ satisfies (S.1)
if and only if $w$ avoids $321$ and $3412$. See Theorem~\ref{thm:toric} and the discussion thereafter.
\end{example}

Recently, the first author and V.~Lakshmibai \cite{Hodges.Lakshmibai:II}  characterized spherical Schubert varieties in the Grassmannian ${\sf Gr}_k({\mathbb C}^n)$. This implies some necessary conditions for a Schubert variety in the flag variety to be spherical. 

Since this work was submitted, Y.~Gao and the authors have proved Conjecture~\ref{conj:main} for type $A$
\cite{Gao.Hodges.Yong}.

\subsection{Summary of the remainder of this paper}
In Section~\ref{sec:2}, we describe some basic properties of Definition~\ref{def:main}. These are used to confirm agreement of Conjecture~\ref{conj:main} in other examples, as well as with geometric properties of Definition~\ref{def:geomain}. Our initial
result is
\begin{itemize}
\item[(I)] Theorem~\ref{prop:agreewithMWZ}, a characterization of when $w_0\in W$ is $I$-spherical. This is connected to \cite{MWZ} and \cite{Stem:Weyl}, supplying some
general-type evidence for Conjecture~\ref{conj:main}.
\end{itemize}
We characterize maximally spherical elements of dihedral groups (Proposition~\ref{prop:dihedral}). This result and (I) are  used to prove:
\begin{itemize}
\item[(II)] Conjecture~\ref{conj:main} holds for rank two simple cases (Theorem~\ref{prop:ranktwo}).
\end{itemize}

In Section~\ref{sec:3}, we turn to $G=GL_n$. We state
\begin{itemize}
\item[(III)] Theorem~\ref{thm:bigrassmannian}, which confirms Conjecture~\ref{conj:main}  for the class of \emph{bigrassmannian}
permutations introduced by A.~Lascoux--M.-P.~Sch\"utzenberger \cite{LS:treillis}. 
\end{itemize}
\begin{itemize}
\item[(IV)] Conjecture~\ref{conj:pattern}, which suggests
Definition~\ref{def:main} is a \emph{pattern avoidance property}. (Since this work was submitted, this
has been proved by C.~Gaetz \cite{Gaetz}, using the aforementioned results of \cite{Gao.Hodges.Yong}.)
\end{itemize}

Section~\ref{sec:4} offers a novel perspective on the sphericality problem in terms  
of  the algebraic framework of \emph{split-symmetric} polynomial theory. The latter interpolates
between symmetric polynomial theory and \emph{asymmetric} polynomial theory.\footnote{Borrowing the terminology of \cite{Pechenik.Searles}.} Within this viewpoint, we discuss a unified notion of \emph{multiplicity-free} problems, and contribute to the subject of \emph{key polynomials}. We present
\begin{itemize}
\item[(V)] Theorem~\ref{thm:mfKey}, which characterizes multiplicity-free key polynomials. This supports some sphericality ideas we propose. 
\end{itemize}
The proof of this result is found in the companion paper \cite{Hodges.Yong}, where we also derive a multiplicity-free
result about the \emph{quasi-key polynomials} of S.~Assaf-D.~Searles \cite{Assaf.Searles}.

Using the fact that these polynomials are characters of \emph{Demazure modules}, as well as
 a result of N.~Perrin \cite{Perrin}, we derive:
 \begin{itemize}
 \item[(VI)] Theorem~\ref{thm:fundamentalRelationship}, which translates the geometric
 sphericality problem to one about \emph{split multiplicity-freeness} of \emph{infinitely} many key polynomials. 
 \end{itemize}
A consequence of (VI) is
\begin{itemize}
\item[(VII)] Theorem~\ref{cor:mfkeyweak}, which gives sufficient conditions, \emph{close}
to those of (V), for a key polynomial to be split multiplicity-free. In comparison to \cite{Hodges.Yong}, the geometric and representation-theoretic input of (VI) allows for a relatively short proof.
\end{itemize}
 Although (V) does not give, \emph{per se}, an algorithm to decide sphericality, we suggest
\begin{itemize} 
\item[(VIII)] Conjecture~\ref{conj:sphericalGeneric}, which asserts that 
checking the ``staircase'' key polynomial suffices. This conjecture reduces to a combinatorial question
about the split symmetry of key polynomials; see Conjecture~\ref{conj:weakerver}, Conjecture~\ref{conj:upone}
and Proposition~\ref{prop:allthesame}. (V) is a solution of this problem in the ``most-split'' 
case. 
\end{itemize}
We exhaustively verified that Conjecture~\ref{conj:main} is mutually consistent with Conjecture~\ref{conj:sphericalGeneric} for $n\leq 6$ (and many larger cases). 

Section~\ref{sec:bigrassmannianpf} is the culmination of the methods developed. We prove Theorem~\ref{thm:bigrassmannian} about bigrassmannian permutations. The argument uses Theorem~\ref{thm:fundamentalRelationship}, a combinatorial formula for splitting key polynomials due to C.~Ross and the second author \cite{Ross.Yong}, as well as an algebraic groups argument (Proposition~\ref{prop:dynkingeo}). 

\section{Basic properties and more examples}\label{sec:2}

Let $\leq$ denote the (strong) \emph{Bruhat order} on $W$. The following is a standard result 
(see, \emph{e.g.}, \cite[Theorem~2.2.2]{Bjorner.Brenti}):

\begin{theorem}[Subword property]\label{thm:subword}
Fix $s_{i_1}s_{i_2}\cdots s_{i_{\ell(v)}}\in {\text{Red}}(v)$.
$u\leq v$ if and only if there exists
$1\leq j_1<j_2<\ldots<j_{\ell(u)}$ such that $s_{i_{j_1}}s_{i_{j_2}}\cdots s_{i_{j_{\ell(u)}}}\in {\text{Red}}(u)$.
\end{theorem}

\begin{proposition}\label{prop:subtrick}
Suppose $v\in W$ and $I\subseteq J(v)$. If there exists $u\in W$ such that $u\leq v$, and every element 
of ${\text{Red}}(u)$ fails (S.1) or (S.2) (with respect to $I$, ignoring whether or not $I\subseteq J(u)$), then $v$ is not $I$-spherical.
\end{proposition}
\begin{proof}
Suppose $v$ is $I$-spherical and 
$R=s_{i_1}\cdots s_{i_{\ell(v)}}\in {\text{Red}}(v)$ 
is an $I$-witness.
Then by Theorem~\ref{thm:subword}, some subexpression $R'$ of $R$ is in ${\text{Red}}(u)$. However, by hypothesis,
$R'$ fails (S.1) or (S.2) with respect to $I$. Hence so must $R$, a contradiction.
\end{proof}

If $W$ is a Weyl group, Bruhat order is the \emph{inclusion order} on Schubert varieties. That is, 
$X_{u}\subseteq X_v \iff u\leq v$.
In particular, $X_{w_0}=G/B$ and $X_{id}=B/B$ is the \emph{Schubert point}. Both of these Schubert varieties
are maximally spherical. In the former case, $H=G$ and in the latter case $H=T$. This is consistent with:

\begin{lemma}\label{lemma:theid}
Both $w=id,w_0$ are maximally spherical.
\end{lemma}
\begin{proof}
If $w=id$, (S.1) is trivial while (S.2) is vacuous (since $J(w)=\emptyset$). If $w=w_0$ then (S.1) is vacuous (since $J(w)=[r]$) while (S.2) is trivial.
\end{proof}

Extending Lemma~\ref{lemma:theid},
we characterize $I$-sphericality of $w_0$. This is a nontrivial confirmation of Conjecture~\ref{conj:main}.

\begin{theorem}[The long element $w_0$]
\label{prop:agreewithMWZ}
Let $n\geq 4$. Suppose $I\subseteq [n-1]$ then $w_0\in {\mathfrak S}_n$ is $I$-spherical if and only if $I=[1,n-1], I=[2,n-1]$ or $I=[1,n-2]$. 

If $W$ is a finite, irreducible Weyl group not of type $A_{n-1}$, then $w_0\in W$ is $I$-spherical if and only if $I=S$.

Hence, Conjecture~\ref{conj:main} holds for all Levi subgroup actions on $G/B$ (where $G$ is simple).
\end{theorem}

\begin{proof}
We first prove the type $A_{n-1}$ statement. 

($\Rightarrow$) (By contrapositive) Assume $I$ is not one of the three listed cases. 

First suppose there exists $2\leq j\leq n-2$ such that $j\not \in I$.  For $1\leq i\leq n-3$, 
let $w_0^{\langle i\rangle}=s_i s_{i+1} s_{i+2} s_{i} s_{i+1} s_{i}\in {\mathfrak S}_n$. So $w_0^{\langle 1\rangle}=
{\color{blue} 4\ 3 \ 2 \ 1} \ 5 \ 6 \ \ldots \ n-2 \ n-1 \ n$, $w_0^{\langle 2\rangle}=1 \ {\color{blue} 5 \ 4 \ 3 \ 2} \ 6 \ 7 \ \ldots n-1 \ n$, \emph{etc.} That is,
each is a ``shifted copy'' of $4321\in {\mathfrak S}_4$. If $n=4$ one checks directly that 
$s_2$ appears twice in any reduced word for $w_0$ (there are sixteen such words). It follows that every $R\in {\text{Red}}(w_0^{\langle j-1\rangle})$
contains $s_j$ twice. Thus $R$ fails (S.1) with respect to $I$. Since $w_0^{\langle j-1\rangle}\leq w_0$, we may apply
Proposition~\ref{prop:subtrick} to conclude $w_0$ is not $I$-spherical.

The remaining possibility is that $I= [2,n-2]$.  Consider $R^c=s_1 s_2\cdots s_{n-1}\in \text{Red}(c)$,
the unique reduced expression for the Coxeter element $c$.
Since $c\leq w_0$, by Theorem~\ref{thm:subword}, for any
$R^{w_0}\in{\text{Red}}(w_0)$, $R^c$ appears as a subexpression of $R^{w_0}$.
In particular, there is an $s_1$ to the left of $s_{n-1}$ in $R^{w_0}$. Now, if
$R^{c'}=s_{n-1} s_{n-2}\cdots s_2 s_1\in \text{Red}(c')$ then by the same reasoning there is an $s_{n-1}$ left of $s_1$ in $R^{w_0}$. Hence either
$s_1$ appears at least twice or $s_{n-1}$ appears at least twice in $R^{w_0}$. Therefore $R^{w_0}$ cannot be an $I$-witness, as it fails (S.1). Thus $w_0$ cannot be $I$-spherical.

($\Leftarrow$) When $I=[n-1]=J(w_0)$, we apply Lemma~\ref{lemma:theid}.
Next we prove $w_0$ is $I$-spherical for $I=[1,n-2]$  (the remaining case is similar). The reduced
expression
\[(s_1 s_2 \cdots s_{n-1})(s_1 s_2\cdots s_{n-2})\cdots (s_1s_2 \cdots s_j) \cdots (s_1)\in {\text{Red}}(w_0)\]
uses $s_{n-1}$ exactly once, and so (S.1) holds. Here ${\mathcal G}_I$ is the $A_{n-2}$ Dynkin diagram. Now
(S.2) requires that ${n\choose 2}-1\leq {n-1\choose 2}+{n-2}$; in fact this holds with equality.

The argument for other types follows from a proof of K.~Fan~\cite{Fan.MO} posted in answer to a question asked on MathOverflow by J.~Humphreys. For the sake of completeness we explicate his argument below.
Let $\Delta=\{\alpha_1,\ldots,\alpha_r\}$ be the simple roots.
\begin{claim}
\label{lemma:whenNegative}
Let $I = [r] - \{d\}$ and $\alpha \in \Phi^+$ with $\alpha = \sum_{i=1}^r a_i \alpha_i$. Suppose $w_0 = w_1 s_d w_2$ for $w_1,w_2 \in W_I$. Then $w_1 s_d w_2$ is a reduced product, \emph{i.e.,} $\ell(w_1s_d w_2)=\ell(w_1)+\ell(s_d)+\ell(w_2)$. Further, if $a_d >0$ and $w_0(\alpha) = -\alpha$, then $w_2(\alpha) = \alpha_d$ and $s_d w_2(\alpha) = -\alpha_d$.
\end{claim}
\begin{proof}
We first show that  $w_1 s_d w_2$ is a reduced product. Since $w_2\in W_I$, $s_d w_2$ is a reduced product.
There exists a reduced expression $R=s_{i_1}\cdots s_{i_n}$ where $n=\ell(w_0)-\ell(s_d w_2)=\ell(w_0)-\ell(w_2)-1$ such that
$w_0=s_{i_1}\cdots s_{i_n} s_d w_2$. Since we assumed $w_0=w_1 s_d w_2$, we conclude that in fact
$R\in \text{Red}(w_1)$. Finally 
$\ell(w_0)=\ell(w_1 s_d w_2)=n+1+\ell(w_2)=\ell(w_1)+\ell(s_d)+\ell(w_2)$, 
as desired.

Let $\beta$ be a root. By definition, 
\begin{equation}
\label{eqn:reflectionformula}
s_{i}(\beta)=\beta-2\frac{(\alpha_i,\beta)}{(\alpha_i,\alpha_i)}\alpha_i
\end{equation} where 
$(\cdot,\cdot)$ is the Euclidean inner product on $V=\text{span}(\Phi)$. Pick $s_{i_1'}\cdots s_{i_{\ell(w_2)}'}\in
\text{Red}(w_2)$. Let $\alpha^{[0]}:=\alpha$ and $\alpha^{[f]}$ the result of applying the rightmost $f$-many reflections of
\[R'=(s_{i_1}\cdots s_{i_n})s_d(s_{i_1'}\cdots s_{i_{\ell(w_2)}'})\in \text{Red}(w_0)\] to $\alpha$ from right to left (\emph{e.g.}, 
$\alpha^{[1]}=s_{i_{\ell(w_2)}'}\alpha$ and $\alpha^{[2]}=s_{i_{\ell(w_2)-1}'}s_{i_{\ell(w_2)}}\alpha$, \emph{etc.}).
$\alpha^{[f]}\in \Phi$ since it is  a basic root-system fact that each reflection permutes $\Phi$.

Let $a_i^{[f]}$ be the coefficient of $\alpha_i$ in $\alpha^{[f]}$. By (\ref{eqn:reflectionformula}), if  $s_j$ is the $f$-th generator
of $R'$ from the right, then
\begin{equation}
\label{eqn:July7qqq}
a^{[f]}_i=a^{[f-1]}_i \text{\ for $i \in [r]-\{j\}$.}
\end{equation}
Since $s_d$ appears exactly once in $R'$, by (\ref{eqn:July7qqq}), the coefficient of $\alpha_d$
changes exactly once, and exactly at the step $f=\ell(w_2)+1$. This implies that first, 
$a_d^{[f-1]}=a_d>0$ and thus
$\alpha^{[\ell(w_2)]}\in \Phi^+$. Second, since $w_0(\alpha)=-\alpha$, it implies $a_d^{[f]}=-a_d<0$. However, since $a_i^{[f]}=a^{[f-1]}_i\geq 0$ for $i\neq d$, $\alpha^{[\ell(w_2)+1]}\in \Phi$ is  possible if and only if $a_i=0$ for $i\neq d$ and $a_d=1$ (recall, $a_d\alpha_d\in \Phi$ if and only if $a_d=\pm 1$, by the axioms of root systems). 
Hence $w_2(\alpha) = \alpha_d$ and $s_d w_2(\alpha) = -\alpha_d$.
\end{proof}

\begin{claim} 
\label{prop:w0singlesimple}
Suppose $W$ is a finite, irreducible Weyl group, not type $A$. Define $I = [r] - \{d\}$. Then $w_0\neq w_1 s_d w_2$ with $w_1,w_2 \in W_I$.
\end{claim}
\begin{proof}
Suppose otherwise. Let $\gamma=\sum_{i=1}^r \alpha_i\in \Phi^+$ and let $\theta$ be the highest root in $\Phi^+$. Outside
of type $A$, $\gamma\neq \theta$~\cite[Section 4.9, Table 1]{H90}. In the case of the exceptional groups, one checks by direct computation that $w_0(\gamma)=-\gamma$ and $w_0(\theta)=-\theta$. In types $B_n$ and $C_n$, as well as $D_n$ for even $n$, $w_0(\alpha)=-\alpha$ for all roots $\alpha$ ~\cite[Chap. VI, \S 4 no. 5,6,8]{B02}. In type $D_n$ for odd $n$, $w_0$ corresponds to the automorphism of the roots which interchanges $\alpha_r$ and $\alpha_{r-1}$, and then negates the result~\cite[Chap. VI, \S 4 no. 8]{B02}. Hence, $w_0(\gamma)=-\gamma$ and $w_0(\theta)=-\theta$. Thus in all cases, $\gamma$ and $\theta$ both satisfy the
hypotheses of Claim~\ref{lemma:whenNegative}. That claim says that $w_2(\gamma)=w_2(\theta)=\alpha_d$.
Hence $w_0(\gamma)=w_0(\theta)$, which is impossible.
\end{proof}

Concluding, if $I \subsetneq [r]$, there exists a $d \in [r] - I$. By Claim~\ref{prop:w0singlesimple}, $w_0$ fails (S.1) for $d$.

In~\cite[Lemma 5.4]{AP14}, R.~S.~Avdeev and A.~V.~Petukhov show that $G / P_J$ is $L_I$-spherical if and only if $G/P_I \times G / P_J$ is $G$-spherical (where the latter action is the diagonal $G$-action). These diagonal spherical actions are classified in type A by P.~Magyar-J.~Weyman-A.~Zelevinsky~\cite{MWZ}. In particular, \cite[Theorem 2.4]{MWZ} shows that $SL_n/B$ is $L_I$-spherical only for the $I$ in the statement of the theorem. The diagonal spherical actions in all other types were given by J.~Stembridge in \cite{STEM,Stem:Weyl}, whose work implies that if $G$ is simple and not of type $A$, then the only Levi that acts spherically on $G/B$ is $G$.
\end{proof}

In our proof of Theorem~\ref{thm:bigrassmannian} we will need the notions from this next example:

\begin{example}[The canonical reduced expression]\label{exa:canonicalred}
The diagram $D(w)$ of $w\in {\mathfrak S}_n$ is the subset of $[n]\times [n]$ given by 
\begin{equation}
\label{eqn:diagram}
D(w)=\{(i,j) \in [n]^2: j<w(i), i<w^{-1}(j)\}
\end{equation}
(in matrix coordinates). Fill the boxes of row $i$ from left to right by $s_i, s_{i+1}, s_{i+2},\ldots$.
Define $R^{\text{canonical}}(w)$ to be the \emph{canonical reduced expression} for $w$ obtained by reading this filling from right to left along rows and from top to bottom.
In ${\mathfrak S}_4$, $w$ is maximally spherical if and only if $R^{\text{canonical}}(w)$ is a $J(w)$-witness for $w$, unless 
$w=3421,4213, 4231$. For instance
$R^{\text{canonical}}(3421)=s_2 s_1 s_3 s_2 s_3$ fails (S.1) when $I=J(3421)=\{1,2\}$. However 
$R=s_1 s_2 s_1 s_3 s_2$ is a $\{1,2\}$-witness in this case.
\end{example}

\begin{proposition}[Dihedral groups]\label{prop:dihedral}
In type $I_2(n)$ and $n\geq 2$ (where $W$ is the the dihedral group of order $2n$), $w\in W$ is maximally spherical if and only if $\ell(w)\leq 3$ or $w=w_0$. 
\end{proposition}
\begin{proof}
The Coxeter diagram is \dynkin[Coxeter, gonality=n]I{}. $W$ is generated by $S=\{s_1,s_2\}$ with the relations
$s_1^2=s_2^2=id$ and $(s_1s_2)^n=id$. Each element of $W$ has a unique reduced word, except
$w_0$. Now $id,w_0$ are maximally spherical by Lemma~\ref{lemma:theid}. Thus suppose $w\neq id,w_0$.
If $w=s_2\cdots$ then $J(w)=\{2\}$. If $\ell(w)\leq 3$ then $w=s_2,s_2s_1$ or $s_2 s_1 s_2$, and it contains
at most one $s_1$, and hence (S.1) is satisfied. (S.2) says there are at most two $s_2$ in the reduced
word of $w$, which is true.  Thus $w$ is $J(w)$-spherical. However, if $4\leq \ell(w)< n$ then $w=s_2 s_1 s_2 s_1 \cdots$
and $w$ contains at least two $s_1$'s, violating (S.1). Thus such $w$ are not $J(w)$-spherical. Similarly, one 
argues the cases where $w=s_1 \cdots$. 
\end{proof}

\begin{corollary}\label{cor:B2G2}
Conjecture~\ref{conj:main} holds for types $B_2$ and $G_2$.
\end{corollary}
\begin{proof}
First let us assume $I=J(w)$. The associated Coxeter groups are dihedral, and hence Proposition~\ref{prop:dihedral} applies. In
type $B_2$ (\dynkin[label, edge length=0.5cm]B2) that proposition states that all $w\in W$ are maximal-spherical.
In type $G_2$ \dynkin[label, edge length=0.5cm]G2, it says that only $id, s_1, s_2, s_1 s_2, s_2 s_1, s_1 s_2 s_1, s_2 s_1 s_2,w_0$ are maximal-spherical. This agrees with the geometric findings of M.~Can and the first author 
\cite[Sections~7,8]{Can.Hodges:2020}.

Thus we may assume $I\subsetneq J(w)$. If $\#I=1$ and $I\subsetneq \{1,2\}$ then $w=w_0$.
In $B_2$, $w_0=s_1 s_2 s_1 s_2 \in W(B_2)$ fails (S.2) and is not $I$-spherical. This agrees with
Theorem~\ref{prop:agreewithMWZ}. Similarly we handle $G_2$. Finally, if $I=\emptyset$, we may appeal to the
toric classification of P.~Karuppuchamy (see Example~\ref{exa:toricJune17}). 
\end{proof}
 
\begin{theorem}[Rank two]\label{prop:ranktwo}
Conjecture~\ref{conj:main} holds for $G/B$ where $G$ is simple of rank two.
\end{theorem}
\begin{proof}
The $B_2$ and $G_2$ cases are covered by Corollary~\ref{cor:B2G2}. 

For the root system $A_2$, first suppose $I=J(w)$. All elements of ${\mathfrak S}_3$ are maximally spherical 
(see Example~\ref{ex:S123456}). Now we apply the results of M.~Can and the first author (Example~\ref{exa:CanHodges}).
If $\#I=1$ and $I\subsetneq J(w)$ then $w=w_0$ and $w=s_1 s_2 s_1\equiv s_2 s_1 s_2$ is $I$-spherical. This agrees
with \cite{MWZ}. Finally if $I=\emptyset$ then we use the  toric classification of P.~Karuppuchamy (see Example~\ref{exa:toricJune17}). 
\end{proof}

We now record facts that infer one kind of sphericality from another. Consistency between the combinatorial
predictions and the geometry are checked.

\begin{proposition}
\label{prop:godown}
Fix $x,y\in W$ with $x\leq y$ and $I\subseteq J(x)\cap J(y)$. If $y$ is $I$-spherical, then $x$ is $I$-spherical.
\end{proposition}
\begin{proof}
The contrapositive claim is Proposition~\ref{prop:subtrick}.
\end{proof}

Proposition~\ref{prop:godown} is consistent with geometry. A normal $H$-variety $Y$ is $H$-spherical if and only if there are finitely many $B_H$-orbits in $Y$ (here $B_H$ is a Borel subgroup of $H$)
  \cite[Theorem 2.1.2]{Perrin}. Now, suppose 
$X$ is a subvariety of $Y$, where $Y$ is $H$-spherical and $X$ is $H$-stable. Then $Y$ must have finitely many $B_H$-orbits, which implies $X$ must have finitely many $B_H$ orbits. Hence, $X$ is $H$-spherical as well. 
In our case, if $x\leq y$ and $I\subseteq J(x)\cap J(y)$ then $H=L_I$ acts
on $X=X_x$ and  $Y=X_y$.

\begin{proposition}[Monotonicity]
\label{prop:monotone}
Let $w\in W$ and 
suppose $I'\subset I\subseteq J(w)$. If $w$ is $I'$-spherical then it is $I$-spherical.
\end{proposition}
\begin{proof}
Suppose $R=s_{r_1}\cdots s_{r_{\ell(w)}}\in {\text{Red}}(w)$ is an $I'$-witness. We show $R$ is an $I$-witness.
Trivially, $R$ satisfies (S.1) with respect to $I$. Let 
\[{\mathcal G}_{I'}=\bigcup_{z'=1}^{m'} {\overline{\mathcal C}}^{(z')} \text{ \ and \ ${\mathcal G}_{I}=\bigcup_{z=1}^m 
{{ {\mathcal C}}}^{(z)}$}\]
be the decomposition (\ref{eqn:thedecompabc}) for $I'$ and $I$, respectively. Suppose $z\in [m]$ is such that
\begin{equation}
\label{eqn:May24abc}
\#\{i_t: i_t \in {\mathcal C}^{(z)}\} > \ell(w_0^{(z)})+\#\text{vertices}({\mathcal C}^{(z)}).
\end{equation}
Let $z_1',z_2',\ldots, z_s' \in [m']$ be such that ${\overline{\mathcal C}}^{(z_j')}\subseteq {{{\mathcal C}}}^{(z)}$.
Let ${\overline w}_0^{(z_j')}$ be the longest element of the Coxeter group $W({\overline{\mathcal C}}^{(z_j')})$
associated to ${\overline{\mathcal C}}^{(z_j')}$, for $1\leq j\leq s$. Now, each  $W({\overline{\mathcal C}}^{(z_j')})$
is a parabolic subgroup of $W({{{\mathcal C}}}^{(z)})$ and \[\prod_{j=1}^s {\overline w}_0^{(z_j')}\leq w_0^{(z)}.\]
Thus, 
$\sum_{j=1}^{s}\ell({\overline w}_0^{(z_j')})\leq \ell(w_0^{(z)})$,
and hence
\begin{equation}
\label{eqn:May24abd}
\sum_{j=1}^{s}\left(\ell({\overline w}_0^{(z_j')})+\#\text{vertices}({\overline{\mathcal C}}^{(z_j')})\right)<\ell(w_0^{(z)})+
\#\text{vertices}({\overline{\mathcal C}}^{(z)}),
\end{equation}
Combining (\ref{eqn:May24abc}), (\ref{eqn:May24abd}) and the pigeonhole principle implies (S.2), with respect to $I'$, fails for some $z_j'$, a contradiction. Thus $R$ satisfies (S.2) with respect to $I$, and therefore $R$ is an $I$-witness.
\end{proof}

Proposition~\ref{prop:monotone} is consistent with the following (known) fact:

\begin{proposition}[Geometric monotonicity]\label{prop:geometricmono}
Suppose $w\in W$ and $I'\subseteq I\subseteq J(w)$. If  $X_w$ is $L_{I'}$-spherical, then 
$X_w$ is $L_I$-spherical.
\end{proposition}
\begin{proof} 

Any Borel subgroup in $L_{I'}$ is of the form $B_{I'} := L_{I'} \cap B$ for some Borel subgroup $B$ of $G$. Then $B_{I'} \subseteq B_I:=L_I \cap B$. 
Clearly if $B_{I'}$ has an open dense orbit in $X_w$, then $B_I$ must have an open dense orbit in $X_w$. Thus if $X_w$ is $L_{I'}$-spherical, then $X_w$ is $L_I$-spherical.
\end{proof}

\begin{Remark}
An anonymous referee points out to us that, in view of 
Proposition~\ref{prop:geometricmono}, the terminology we use of $X_w$ being
``maximally spherical'' if it is $L_{J(w)}$ spherical is, in a sense, backwards. By Proposition~\ref{prop:geometricmono}, $X_w$ being $L_{J(w)}$-spherical is a necessary condition for
it to be $L_I$-spherical for any $I\subsetneq J(w)$. Hence, $L_{J(w)}$-spherical is ``least spherical'', and the ``most spherical'' are those that are $L_{\emptyset}$-spherical since they are $L_{I'}$-spherical for \emph{any} $I'\subseteq J(w)$. Due to Proposition~\ref{prop:monotone}, a similar remark applies to our notion of $w$ being ``maximally spherical''. 
\end{Remark}

\begin{proposition}
\label{prop:twopieces}
Suppose $X, Y\subseteq [r]$ where $[s_x,s_y]=id$ for all $x\in X, y\in Y$.
Let $w=uv$ where $u\in W_X$ and $v\in W_Y$. If $I\subseteq J(w)$
then $w$ is an $I$-spherical element of $W$ if and only if $u$ is an $(I\cap X)$-spherical element of 
$W_X$ and $v$ is an $(I\cap Y)$-spherical element of $W_Y$.
\end{proposition}
\begin{proof}
This follows since $J(u)=J(w)\cap X$ and $J(v)=J(w)\cap Y$, and since any component of
${\mathcal G}_I$ is a component of the induced subdiagram of ${\mathcal G}_X$ on the nodes ${I\cap X}$ or the induced subdiagram of ${\mathcal G}_Y$ on the nodes ${I\cap Y}$.
\end{proof}

Suppose $D,D'$ are two Coxeter diagrams and $\phi:D\hookrightarrow D'$ is an embedding of Coxeter diagrams
(preserving edge multiplicities). Then $\phi$ induces an embedding of Coxeter groups $(W_D,S_D)\hookrightarrow
(W_{D'},S_{D'})$, their labellings $[r_D]\hookrightarrow [r_{D'}]$, and root systems $(\Phi_D,\Delta_D)\hookrightarrow (\Phi_{D'},\Delta_{D'})$. Abusing notation, we use
$\phi$ to indicate all of these injections.

\begin{proposition}[Diagram embedding]
\label{prop:Dynkininjection}
If $w\in W_D$ is $I$-spherical for $I\subseteq J(w)$ then $\phi(w) \in W_{D'}$ is $\phi(I)$-spherical.
\end{proposition}
\begin{proof}
Suppose $R=s_{i_1}\cdots s_{i_{\ell(w)}}\in \text{Red}_{(W_D,S_D)}(w)$ is an $I$-witness.
We may suppose that
the $\phi$ sends $D$ to the nodes of $D'$ labelled by $1',2',\ldots, r_D'$. Then
$s_{i_1'}\cdots s_{i_{\ell(w)}'}\in \text{Red}_{(W_{D'},S_{D'})}(\phi(w))$ 
and clearly 
\[\phi(I)\subseteq \phi(J(w))=J(\phi(w))\] 
(thus it make sense
to ask if $\phi(w)$ is $\phi(I)$-spherical). Since 
$\phi([r_D]- I)=\{1',2',\ldots, r_D'\}- \phi(I)$, 
(S.1) holds for $\phi(I)$.
Now (S.2) holds since ${\mathcal G}_I\cong {\mathcal G}_{\phi(I)}$ (Coxeter diagram isomorphism).
\end{proof}

\begin{example}[$D_4$] Of the $2^3 4!=192$ many elements of the Weyl group of type \dynkin[labels={1,3,2,4}, edge length=0.5cm]D4, the $38$ that are not $J(w)$-spherical are given in Table~\ref{table:D4}. One can check that
the list is consistent with Propositions~\ref{prop:twopieces} and~\ref{prop:Dynkininjection}. For instance, from Example~\ref{ex:S123456}, all elements of the Weyl groups for $A_1,A_2,$ and $A_3$ are maximally spherical. This
combined with the two propositions says that any $w\in W(D_4)$ that is in a (strict) parabolic subgroup is spherical.
That is why all of the words in the table use the entirety of $S$.
\end{example}

\begin{table}[]
\begin{tabular}{lll}
$s_1 s_2 s_3 s_1 s_4 s_3 s_1$ &  $s_1 s_2 s_3 s_1 s_2 s_4 s_3 s_1$ &   $s_1 s_3 s_1 s_2 s_3 s_4 s_3 s_1$\\
  $s_1 s_2 s_3 s_1 s_2 s_3 s_4 s_3 s_1$ &
$s_1 s_2 s_3 s_1 s_4 s_3 s_1 s_2$ &
  $s_2 s_3 s_1 s_2 s_4 s_3 s_1 s_2$\\
    $s_1 s_2 s_3 s_1 s_2 s_4 s_3 s_1 s_2$ &
$s_2 s_3 s_2 s_4 s_3 s_1 s_2$ &
$s_1 s_3 s_1 s_2 s_3 s_4 s_3 s_1 s_2$\\
  $s_2 s_3 s_1 s_2 s_3 s_4 s_3 s_1 s_2$ & 
    $s_1 s_2 s_3 s_1 s_2 s_3 s_4 s_3 s_1 s_2$ &  $s_2 s_3 s_1 s_2 s_4 s_3 s_2$\\
    $s_1 s_2 s_3 s_1 s_2 s_4 s_3 s_2$ & $s_2 s_3 s_2 s_4 s_3 s_2$ &  $s_1 s_3 s_1 s_2 s_3 s_4 s_3 s_2$\\
    $s_2 s_3 s_1 s_2 s_3 s_4 s_3 s_2$ &
  $s_1 s_2 s_3 s_1 s_2 s_3 s_4 s_3 s_2$ &
   $s_1 s_2 s_3 s_1 s_4 s_3 s_1 s_2 s_3$\\
     $s_1 s_3 s_1 s_2 s_4 s_3 s_1 s_2 s_3$ & 
 $s_2 s_3 s_1 s_2 s_4 s_3 s_1 s_2 s_3$ &   
  $s_1 s_2 s_3 s_1 s_2 s_4 s_3 s_1 s_2 s_3$\\
   $s_3 s_2 s_4 s_3 s_1 s_2 s_3$ &  $s_2 s_3 s_2 s_4 s_3 s_1 s_2 s_3$ &
    $s_3 s_1 s_2 s_3 s_4 s_3 s_1 s_2 s_3$ \\ 
    $s_1 s_3 s_1 s_2 s_3 s_4 s_3 s_1 s_2 s_3$ & 
    $s_2 s_3 s_1 s_2 s_3 s_4 s_3 s_1 s_2 s_3$ &  $s_1 s_2 s_3 s_1 s_2 s_3 s_4 s_3 s_1 s_2 s_3$\\
    $s_1 s_3 s_1 s_2 s_3 s_4 s_3$ &
 $s_1 s_2 s_3 s_1 s_4 s_3 s_1 s_2 s_3 s_4$ &
    $s_1 s_3 s_1 s_2 s_4 s_3 s_1 s_2 s_3 s_4$\\
      $s_2 s_3 s_1 s_2 s_4 s_3 s_1 s_2 s_3 s_4$ & 
    $s_1 s_2 s_3 s_1 s_2 s_4 s_3 s_1 s_2 s_3 s_4$ &
     $s_3 s_2 s_4 s_3 s_1 s_2 s_3 s_4$\\
    $s_2 s_3 s_2 s_4 s_3 s_1 s_2 s_3 s_4$ & $s_3 s_1 s_2 s_3 s_4 s_3 s_1 s_2 s_3 s_4$ & 
    $s_1 s_3 s_1 s_2 s_3 s_4 s_3 s_1 s_2 s_3 s_4$ \\  
    $s_2 s_3 s_1 s_2 s_3 s_4 s_3 s_1 s_2 s_3 s_4$ &
    $s_1 s_2 s_3 s_1 s_2 s_3 s_4 s_3 s_1 s_2 s_3 s_4$ & \ 
\end{tabular}
\smallskip
\caption{Non $J(w)$-spherical elements of $D_4$ \label{table:D4}}
\end{table}

Proposition~\ref{prop:Dynkininjection} is consistent with Conjecture~\ref{conj:main}. In our proof of Theorem~\ref{thm:bigrassmannian}, we will require the geometric version of Proposition~\ref{prop:Dynkininjection} for the general linear group; this is Proposition~\ref{prop:dynkingeo} which we prepare for now. The result holds for reductive groups in other types. We omit the general proof as the algebraic groups setup required is substantial.

Let $n, f, N \in \mathbb{Z}_{>0}$ be such that $n+f \leq N$. We now define maps between the root systems, Weyl groups, and labelings of $GL_n$ and $GL_N$. Let $T_n$ and $T_N$ be the subspaces of diagonal matrices in $GL_n$ and $GL_N$, respectively.
\begin{equation}
\begin{array}{ccccc}
\Phi_n \hookrightarrow \Phi_N & \quad & W_n \hookrightarrow W_N & \quad & [n-1] \hookrightarrow [N-1] \\
\alpha_i \mapsto \alpha_{f+i} & \quad &  s_i \mapsto s_{f+i} & \quad &  i \mapsto f + i \\
\end{array}
\end{equation}
Abusing notation, we use $\iota$ to indicate all of these maps. Let $h: GL_n \hookrightarrow GL_N$ be given by
\begin{equation}
\label{eqn:Jul13vcc}
g \mapsto \begin{bmatrix}
    {\sf Id}_f & & \\
    & g & \\
    & & {\sf Id}_{N-n-f}
  \end{bmatrix},
\end{equation}
where ${\sf Id}_k$ is the $k \times k$ identity matrix. The map $h$ is compatible with the maps $\iota$. That is, 
\begin{equation}
\label{eqn:hcompat}
h(w) = \iota(w)\text{ for $w \in W_n$;}
\end{equation}
here we abuse notation and write $h(w)$ to mean the image, under $h$, of a coset representative of $w$ in $N(T_n)$ is equal to a coset representative of $\iota(w)$ in $N(T_N)$. Further, $h(U_{\alpha}) = U_{\iota(\alpha)}$ where $U_{\alpha}$ is the \emph{root subgroup} of $\alpha \in \Phi_n$.
Since $h(B_n) \subseteq B_N$, $h$ descends to an injective map 
\begin{equation}
\label{eqn:hoverlinewelldefined}
\overline{h}: GL_n / B_n \hookrightarrow GL_N / B_N.
\end{equation}

We now prove a lemma inspired by E.~Richmond and W.~Slofstra's \cite[Lemma~4.8]{Richmond.Slofstra}.
\begin{lemma}
\label{lemma:isoequivariantL}
The map $\overline{h}: GL_n / B_n \hookrightarrow GL_N / B_N$ induces a $L_{J(w)}$-equivariant isomorphism $X_{wB_n} \hookrightarrow X_{\iota(w)B_N}$ for all $w \in W_n$ (the action of $L_{J(w)}$ on the right hand side is $h(L_{J(w)})$).
\end{lemma}
\begin{proof}
That $\overline{h}: X_{wB_n} \hookrightarrow X_{\iota(w)B_N}$ follows from \eqref{eqn:hcompat} and the Bruhat decomposition. Thus, since $X_{\iota(w)B_N}$ is normal, to show that $X_{wB_n} \hookrightarrow X_{\iota(w)B_N}$ is an isomorphism we need only show surjectivity (by Zariski's Main Theorem). 

Let $K = \{f+1, \ldots, f+n-1 \}$. The parabolic $P_K = L_K U_K$, where $U_K=R_u(P_K)$ is the unipotent radical of $P_K$. Let $B_K := L_K \cap B_N$ be a Borel subgroup of $L_K$. From, \emph{e.g.}, the proof of \cite[Lemma~4.8]{Richmond.Slofstra}, we recall that 
\begin{equation}
\label{eqn:borelxunipotent}
B_N = B_K U_K,
\end{equation}
and that $U_K$ is stable under conjugation by any $v \in (W_N)_K$ (parabolic subgroup), and in particular
\begin{equation}
\label{eqn:stableconj}
v^{-1} U_K v B_N = B_N.
\end{equation}

An element $b \in B_K$ has the form 
\begin{center}
$\begin{bmatrix}
    r & & \\
    & s & \\
    & & t
  \end{bmatrix},$
\end{center}
where $r \in T_f$, $s \in B_n$, and $t \in T_{N-n-f}$ (where $T_k$ denotes the subspace of diagonal matrices in $GL_k$). Thus for any such $b$, there exists a 
\begin{center}
$t_b = \begin{bmatrix}
    r^{-1} & & \\
    & {\sf Id}_n & \\
    & & t^{-1}
  \end{bmatrix} \in H := \left\{ \begin{bmatrix}
    A & & \\
    & {\sf Id}_n & \\
    & & C
  \end{bmatrix} : A \in T_f,  B \in T_{N-n-f} \right\}$ 
\end{center}
such that 
\begin{center}
\label{eqn:Jul13ret}
$b t_b = \begin{bmatrix}
    {\sf Id}_f & & \\
    & s & \\
    & & {\sf Id}_{N-n-f}
  \end{bmatrix} = h(s).$
\end{center}
This allows us to conclude that
\begin{equation}
\label{eqn:imgxtorus}
h(B_n) H = B_K.
\end{equation}
Also, notice that
\begin{equation}
\label{eqn:Hcommutes}
H v = v H\text{ for $v \in (W_N)_K$}.
\end{equation}

Consider the Schubert cell of $v \in (W_N)_K$. We have $v = \iota(w)$ for some $w \in W_n$.
\begin{equation*}
\begin{array}{rlr}
B_N v B_N / B_N & = B_K U_K v B_N / B_N & \qquad \qquad \qquad \qquad \qquad \qquad \eqref{eqn:borelxunipotent} \\
& = B_K (v v^{-1}) U_K v B_N / B_N & \\
& = B_K v B_N / B_N & \eqref{eqn:stableconj} \\
& = h(B_n) H v B_N / B_N & \eqref{eqn:imgxtorus} \\
& = h(B_n) v B_N / B_N & \eqref{eqn:Hcommutes}\text{ and }H \subseteq  B \\
& = h(B_n) \iota(w) B_N / B_N & \\
& = h(B_n w) B_N / B_N & \eqref{eqn:hcompat} \\
& = \overline{h}(B_n w B_n / B_n) & \eqref{eqn:hoverlinewelldefined} \\
\end{array}
\end{equation*}
Thus $\overline{h}$ induces a surjection from the Schubert cell of $w\in W_n$ onto the Schubert cell of $\iota(w)=v$. Since the same holds for all $u = \iota(w') \leq \iota(w) = v \in (W_N)_K$, the Bruhat decomposition implies $X_{wB_n} \hookrightarrow X_{\iota(w)B_N}$ is surjective. 

The map $\overline{h}$ is $GL_n$-equivariant (where the action on the right hand side is given by $h(GL_n)$). Thus $L_{J(w)} \subseteq {\rm stab}_{GL_n}(X_{wB_n})$ implies $h(L_{J(w)}) \subseteq {\rm stab}_{GL_N}(X_{\iota(w)B_N})$. Thus the isomorphism $X_{wB_n} \hookrightarrow X_{\iota(w)B_N}$ is $L_{J(w)}$-equivariant. 
\end{proof}

\begin{proposition}[Diagram embedding; geometric version]
\label{prop:dynkingeo}
If $X_{wB_n} \subseteq GL_n / B_n$ is $L_I$-spherical for $I \subseteq J(w)$, then $X_{\iota(w)B_N} \subseteq GL_N / B_N$ is $L_{\iota(I)}$-spherical.
\end{proposition}
\begin{proof}
Lemma~\ref{lemma:isoequivariantL} implies that $X_{wB_n} \cong X_{\iota(w)B_N}$ as $L_{J(w)}$-varieties (and hence as $L_I$-varieties for $I \subseteq J(w)$). If $I \subseteq [n-1]$, then $h(L_I) \subseteq L_{\iota(I)}$. In particular, since $\iota(J(w)) = J(\iota(w))$, this implies \[ h(L_{I}) \subseteq L_{\iota(I)} \subseteq L_{\iota(J(w))} = L_{J(\iota(w))} \subseteq {\rm stab}_{GL_N}(X_{\iota(w)B_N}). \] We conclude that, if $X_{wB_n}$ is $L_{I}$-spherical, then $X_{\iota(w)B_N}$ is $L_{I}$-spherical, which in turn implies  $X_{\iota(w)B_N}$ is $L_{\iota(I)}$-spherical.
\end{proof}

\section{The general linear group}\label{sec:3}

In what remains, $G=GL_n$. This is type $A_{n-1}$, hence $S=\{s_i=(i \ i+1): 1\leq i\leq n-1\}$. We express $w\in W(A_{n-1})\cong {\mathfrak S}_n$ in one-line notation. Here,
\begin{equation}
\label{eqn:Jwright}
J(w)=\{j\in [n-1]: w^{-1}(j)>w^{-1}(j+1)\}
\end{equation}
($j\in J(w)$ if $j+1$ appears to the left of $j$ in $w$'s one-line notation). Indeed, the description (\ref{eqn:Jwright})
is saying the that left descents of $w$ are the \emph{right} descents of $w^{-1}$. Let $I\in 2^{J(w)}$ and
\[D:=[n-1]- I=\{d_1<d_2<d_3<\ldots<d_k\}.\] 
By convention, $d_0:=0,d_{k+1}:=n$.

\begin{definition}[$GL_n$-version of Definition~\ref{def:main}]
$w\in {\mathfrak S}_n$ is \emph{$I$-spherical} 
if $R=s_{i_1}s_{i_2}\cdots s_{i_{\ell(w)}}\in {\text{Red}}(w)$ exists such that
\begin{itemize}
\item[(S.1')] $s_{d_i}$ appears at most once in $R$
\item[(S.2')] $\#\{m:d_{t-1}<i_m< d_t\} < {d_{t}-d_{t-1}+1\choose 2}$   for $1\leq t\leq k+1$.
\end{itemize}
$w$ is \emph{maximally spherical} if it is $J(w)$-spherical.
\end{definition}
Clearly (S.1') is the specialization of (S.1). For (S.2), the Coxeter graph induced by the nodes of
the $A_{n-1}$ diagram strictly between $d_{t-1}$ and $d_t$ is type $A_{d_t-d_{t-1}-1}$. In 
type $A_{d_t-d_{t-1}-1}$, $\ell(w_0)={d_t-d_{t-1}\choose 2}$. Now 
$\ell(w_0)+(d_t-d_{t-1}-1)=
{d_t-d_{t-1}+1\choose 2}-1$,
which agrees with (S.2'), once one accounts for the strict inequality used.

Let ${T}$ be invertible
diagonal matrices and ${ B}$ be the
invertible upper triangular matrices in $G={GL}_n$.  Hence ${G}/{B}$ is the variety ${\sf Flags}({\mathbb C}^n)$ of complete flags of subspaces in ${\mathbb C}^n$. Here, ${ L}_I$ is the Levi subgroup of invertible block matrices 
\begin{equation}
\label{eqn:thechoice}
{L}_I={GL}_{d_1-d_0}\times {GL}_{d_2-d_1}\times \cdots\times
{GL}_{d_k-d_{k-1}}\times {GL}_{d_{k+1}-d_k}.
\end{equation}

\begin{conjecture}[$GL_n$-version of Conjecture~\ref{conj:main}]\label{conj:MainA}
Let $I\subseteq J(w)$. $X_w$ is $L_I$-spherical if and only if $w$ is $I$-spherical.
\end{conjecture}

\begin{example}
Let $w=35246781\in {\mathfrak S}_8$. Here $J(w)=\{1,2,4\}$. If $I=J(w)$ then $D=\{3,5,6,7\}$. 
Now, $R=s_1 s_2 s_1 s_3 s_4 s_3 s_2 s_5 s_6 s_7 \in \text{Red}(w)$, but it fails (S.1'). Instead
consider
\[R'=s_1 s_2 s_1 s_4 s_3 s_2 s_4 s_5 s_6 s_7 \in \text{Red}(w).\] 
(S.1') holds. To verify (S.2') we check that
\begin{itemize}
\item $4=\#\{m:0<r_m< 3\}< {3-0+1\choose 2}=6$
\item $2=\#\{m:3<r_m< 5\}< {5-3+1\choose 2}=3$
\item $0=\#\{m:5<r_m< 6\}< {6-5+1\choose 2}=1$
\item $0=\#\{m:6<r_m< 7\}< {7-6+1\choose 2}=1$
\item $0=\#\{m:7<r_m< 8\}< {8-7+1\choose 2}=1$
\end{itemize}
Hence $w$ is maximally spherical.
\end{example}

\begin{example}
\label{exa:24531predict}
Let $n=5, w=24531$. Here $I=J(w)=\{1,3\}$ and $D=\{2,4\}$. 
Let $R=s_3 s_1 s_2 s_3 s_4 s_3\in \text{Red}(w)$. $R$ satisfies (S.1') but fails (S.2') since 
$\#\{m:2<i_m< 4\}={4-2+1\choose 2}=3$. One checks no expression in $\text{Red}(w)$ is an $I$-witness.  Hence Conjecture~\ref{conj:MainA} predicts
that $X_{24531}$ is not $L_{J(w)}$-spherical. We will \emph{prove} this is true in Example~\ref{exa:May20ddd}. 
\end{example}

A permutation $w\in {\mathfrak S}_n$ is \emph{bigrassmannian} if both $w$ and $w^{-1}$ have a unique descent.
A.~Lascoux-M.-P.~Sch\"utzenberger \cite{LS:treillis} initiated the study of these permutations and identified a number of
their nice (Bruhat) order-theoretic properties.\footnote{For example, $w\in {\mathfrak S}_n$ is bigrassmannian if and
only if it is join-irreducible.} 
The \emph{code} of $w\in {\mathfrak S}_n$, 
\[{\text{code}}(w)=(c_1,c_2,\ldots,c_n),\] 
is defined by letting $c_i$ be the number of boxes in the
$i$-th row of $D(w)$ (as defined in (\ref{eqn:diagram})). In fact, $w$ is bigrassmannian if and only if its diagram
consists of an $b\times a$ rectangle. More precisely, ${\text{code}}(w)=(0^f,a^b,0^g)$ where $f+b+g=n$. 

For later reference, we record a simple (and well-known) observation:
\begin{lemma}
\label{lemma:bigrassfact}
If $w$ is bigrassmannian with ${\text{code}}(w)=(0^f,a^b,0^g)$ where $f+b+g=n$, then the unique descent
of $w$ is at position $f+b$, the unique descent of $w^{-1}$ is at $f+a$, and in particular $J(w)=\{f+a\}$.
Moreover, $f+a+1$ appears left of $f+a$ in $w$'s one-line notation.
\end{lemma}
\begin{proof} 
The first sentence follows from elementary considerations about $D(w)$ (defined in (\ref{eqn:diagram})); see \cite[Section~2.1]{Manivel}
and more specifically \cite[Proposition~2.1.2]{Manivel}. The second sentence is the
parenthetical immediately after (\ref{eqn:Jwright}), for the case at hand.
\end{proof}

\begin{theorem}\label{thm:bigrassmannian}
Let $w\in {\mathfrak S}_n$ be bigrassmannian.
Conjecture~\ref{conj:MainA} holds for $I=J(w)$. Moreover, $w$ is $J(w)$-spherical if and only if 
\begin{equation}
\label{eqn:3cases}
{\text{code}}(w)\in\{(0^f,a,0^g), (0^f,1^b,0^g), (0^f,2^2,0^g)\}.
\end{equation}
\end{theorem}

When $w$ is bigrassmannian, $\#J(w)=1$. Thus, the remaining bigrassmannian case
of Conjecture~\ref{conj:main} (equivalently, Conjecture~\ref{conj:MainA})
not covered in the statement of Theorem~\ref{thm:bigrassmannian}  is $I=\emptyset$. However, that case
is covered by the toric classification of P.~Karuppuchamy (see Example~\ref{exa:toricJune17}).
We will delay the proof of Theorem~\ref{thm:bigrassmannian} until Section~\ref{sec:bigrassmannianpf}, after
building up the framework used for the proof.

\begin{example}
A permutation $w\in {\mathfrak S}_n$ is \emph{dominant} if ${\text{code}}(w)$ is a partition. 
For $n=5$, the codes of the non $J(w)$-spherical dominant permutations are:
\begin{align*}
 (2, 2, 2, 0, 0), (2, 2, 2, 1, 0), (3, 3, 0, 0, 0), (3, 3, 1, 0, 0), (3, 3, 1, 1, 0)\\
  (3, 3, 2, 0, 0), (4, 1, 1, 0, 0), (4, 1, 1, 1, 0), (4, 2, 0, 0, 0), (4, 2, 2, 0, 0)\\
   (4, 2, 2, 1, 0), (4, 3, 0, 0, 0), (4, 3, 1, 0, 0), (4, 3, 1, 1, 0)
\end{align*}
What is the general classification of these partitions? In M.~Develin-J.~Martin-V.~Reiner's \cite{ReinerDing},
the associated $X_w$ are called
\emph{Ding's Schubert varieties} (in reference to K.~Ding's \cite{Ding}). Hence we are asking which of Ding's Schubert varieties are $L_{J(w)}$-spherical (and more generally, one can ask which of these Schubert varieties
are $L_I$-spherical).
\end{example}

We expect that Schubert varieties $X_w$ are rarely $L_{J(w)}$-spherical. Theorem~\ref{thm:bigrassmannian}
gives some concrete indication of this assertion. In view of Conjecture~\ref{conj:MainA}, we believe the following
enumerative assertion is true:

\begin{conjecture}\label{conj:vanishing}
$\lim_{n\to \infty} \#\{w\in {\mathfrak S}_n: w \text{ \ is $J(w)$-spherical}\}/n!\to 0$.\footnote{Since this work was
submitted, Conjecture~\ref{conj:vanishing} has been proved in work of D.~Brewster and the authors
\cite{Brewster.Hodges.Yong}.}
\end{conjecture}

(Conjecture~\ref{conj:vanishing} should also hold for other Weyl groups of classical type.)

Suppose $u\in {\mathfrak S}_n$ and $v\in {\mathfrak S}_N$. Let $u\hookrightarrow v$ denote a \emph{pattern embedding}, \emph{i.e.}, there exists $\phi_1<\phi_2<\ldots <\phi_n$ such that $v(\phi_1),\ldots,v(\phi_n)$ are in the same relative
order as $u(1),\ldots,u(n)$. One says $v$ \emph{avoids} $u$ if no such embedding exists.
 
\begin{conjecture}[Pattern avoidance]\label{conj:pattern}
If $u\in {\mathfrak S}_n$ is not $J(u)$-spherical and $u\hookrightarrow v\in {\mathfrak S}_N$ ($N>n$) then
$v\in {\mathfrak S}_N$ is not $J(v)$-spherical. Moreover, the complete list of bad patterns are the not maximally spherical
elements of ${\mathfrak S}_5$ (listed in Example~\ref{ex:S123456}).\footnote{As mentioned in the Introduction,
Conjecture~\ref{conj:pattern} has since been proved by C.~Gaetz \cite{Gaetz}.}
\end{conjecture}

With the assistance of J.~Hu, we verified that all bad cases in ${\mathfrak S}_n$ for $n\leq 7$ 
can be blamed on the ${\mathfrak S}_5$ patterns. 
It seems plausible to attack this problem by extending the ideas in Section~\ref{sec:bigrassmannianpf}. 
We hope to return to this in future work.

\section{Polynomials}\label{sec:4}

We formalize a ``split-symmetry'' framework on algebraic combinatorics of polynomials in order to study the Levi sphericality problem. 

\subsection{Split-symmetry in algebraic combinatorics}
Algebraic combinatorics has, at its core, the
study of elements/bases of the ring of symmetric polynomials ${\sf Sym}(n)$ (see, e.g., \cite[Chapter~7]{ECII}). Obversely,
A.~Lascoux--M.-P.~Sch\"utzenberger introduced numerous asymmetric families in the polynomial ring ${\sf Pol}(n)$;
see, \emph{e.g.}, \cite{Lascoux:polynomials, Pechenik.Searles} and the references therein. We now discuss an interpolation between ${\sf Sym}(n)$ and ${\sf Pol}(n)$:

\begin{definition}[Split-symmetry]
\label{def:split}
Fix integers 
$d_0:=0<d_1<d_2<\ldots<d_k < d_{k+1} := n$ 
with $D := \{d_1,\ldots,d_k\}$.
${\Pi}_D$ is
the subring of ${\sf Pol}(n)$ consisting 
of polynomials separately symmetric in 
$X_i:=\{x_{d_{i-1}+1},\ldots,x_{d_i}\} \text{\ for $1\leq i\leq k+1$.}$ 
A polynomial is $D$-\emph{split-symmetric} if $f\in {\Pi}_D$. 
\end{definition}

Clearly,
\begin{proposition}
\label{prop:thesplit}
${\Pi}_D \cong {\sf Sym}(d_1)\otimes {\sf Sym}(d_2-d_1)\otimes \cdots\otimes 
{\sf Sym}(d_{k+1}-d_{k})$.
\end{proposition}

A \emph{partition} of length $n$ is a sequence $\lambda=(\lambda_1,\ldots,
\lambda_{n})$ of non-negative integers with $\lambda_1 \geq \cdots \geq \lambda_n$.
Let ${\sf Par}_n$ be the set of such partitions. The \emph{Schur polynomial} is 
$s_{\lambda}=\sum_{T} x^T$,
where the sum is over semistandard Young tableaux of shape $\lambda$ with entries from $[n]$. Here,  
$x^T:=\prod_{i=1}^n x_i^{\#i\in T}$.
The set $\{ s_{\lambda}(x_1,\ldots,x_n):\lambda\in {\sf Par}_n\}$ 
is a ${\mathbb Z}$-linear basis of ${\sf Sym}(n)$.

\begin{definition}
The \emph{$D$-Schur polynomials} are
$s_{\lambda^{1},\ldots,\lambda^{k}} := s_{\lambda^{1}}(X_1) s_{\lambda^{2}}(X_2) \cdots s_{\lambda^{k}}(X_k)$,
where $ (\lambda^{1},\ldots,\lambda^{k}) \in 
{\sf Par}_D := {\sf Par}_{d_{1} - d_{0}} \times \cdots \times {\sf Par}_{d_{k+1} - d_{k}}$.
\end{definition}
By Proposition~\ref{prop:thesplit}, and the basis property of (classical) Schur polynomials, we have
\begin{corollary}\label{cor:splitschur}
$\{s_{\lambda^{1},\ldots,\lambda^{k}} : (\lambda^{1},\ldots,\lambda^{k}) \in 
{\sf Par}_D\}$ forms a basis of ${\Pi}_D$.
 \end{corollary}
 
\subsection{Key polynomials}\label{sec:keysxyz}
The \emph{Demazure operator} is
\begin{align*}
\pi_j &:{\sf Pol}_n\to {\sf Pol}_n\\
& \ \ f\mapsto \frac{x_j f - x_{j+1}s_jf}{x_j-x_{j+1}},
\end{align*}
where ${s_j}f:=f(x_1,\ldots,x_{j+1},x_j,\ldots,x_n)$. 

A \emph{weak composition} of length $n$ is a sequence $\alpha=(\alpha_1,\ldots,
\alpha_{n})\in {\mathbb Z}_{\geq 0}^n$.
Let ${\sf Comp}_n$ denote the set of these weak compositions. 
Given $\alpha \in {\sf Comp}_n$, the \emph{key polynomial} $\kappa_{\alpha}$ is
\[x^{\alpha}:=x_1^{\alpha_1}\cdots x_n^{\alpha_n}, \mbox{\ \ \ if $\alpha$ is weakly decreasing.}\]
Otherwise, set
\begin{equation}
\label{eqn:keypidef}
\kappa_{\alpha}=\pi_j(\kappa_{\widehat \alpha}) \mbox{\ where $\widehat\alpha=(\alpha_1,\ldots,\alpha_{j+1},\alpha_j,\ldots,\alpha_n)$ and
$\alpha_{j+1}>\alpha_{j}$.}
\end{equation}
The key polynomials for $\alpha \in {\sf Comp}_n$
form a ${\mathbb Z}$-basis of ${\mathbb Z}[x_1,\ldots,x_n]$; see work of V.~Reiner--M.~Shimozono \cite{Reiner.Shimozono} (and references therein) for more on $\kappa_{\alpha}$. Since it is known that
the $\pi_j$ operators satisfy the commutation relations $\pi_i \pi_j = \pi_j \pi_i$ (for $|i-j|>1$) and the braid
relations $\pi_i \pi_{i+1}\pi_i = \pi_{i+1}\pi_i \pi_{i+1}$ (for $1\leq i\leq n-1$), the above recurrence is well-defined.

Define a \emph{descent} of a composition $\alpha$ to be an index $j$ where $\alpha_j > \alpha_{j+1}$.
Let ${\sf Comp}_n(D)$ be those $\alpha\in {\sf Comp}_n$ with descents contained in $D=\{d_1,\ldots,d_k\}$ with $d_1<\ldots<d_k$.

Although we will not need it in this paper, let us take this opportunity to prove:
\begin{proposition}
\label{prop:ZBasis}
$\{\kappa_{\alpha}: \alpha\in {\sf Comp}_n(D)\}$ forms a  ${\mathbb Z}$-linear basis of ${\Pi}_D$.
\end{proposition}
\begin{proof}
If $d_i\leq j<d_{i+1}$,
then $\pi_j(\kappa_{\alpha})=\pi_j(\pi_j(\kappa_{\widehat\alpha}))=\pi_j(\kappa_{\widehat\alpha})=\kappa_{\alpha} $ (since it is also true that $\pi_j^2=\pi_j$). Thus,
\begin{align*}
\kappa_{\alpha}=\frac{x_{j}\kappa_{\alpha}-x_{j+1}s_j\kappa_{\alpha}}{x_j-x_{j+1}}
& \iff (x_j-x_{j+1})\kappa_{\alpha}=x_{j}\kappa_{\alpha}-x_{j+1}s_j\kappa_{\alpha}\\
& \iff (\kappa_{\alpha}-s_j\kappa_{\alpha})x_{j+1}=0\\
& \iff \kappa_{\alpha}=s_j\kappa_{\alpha}.
\end{align*}
Hence $\kappa_{\alpha}\in \Pi_D$. Suppose a nonzero $g\in \Pi_D$ is given. By Corollary~\ref{cor:splitschur},
\[g=\sum_{\lambda^{1},\lambda^{2},\ldots,\lambda^{k}}c_{\lambda^{1},\lambda^{2},\ldots,\lambda^{k}} s_{\lambda^{1},\ldots,\lambda^{k}},\]
where each $c_{\lambda^{1},\lambda^{2},\ldots,\lambda^{k}}$ is a scalar and 
$(\lambda^{1},\lambda^{2},\ldots,\lambda^{k})\in {\sf Par}_D$.

Let $\overline{\lambda^{i}}$ be the parts of $\lambda^{i}$ be written in non-decreasing order (\emph{i.e.}, a ``reverse partition''). Then let $\alpha= \overline{\lambda^{1}},\ldots,
\overline{\lambda^{k}} \in {\sf Comp}_n$ be obtained as 
the concatenation of these reverse partitions. Thus, 
$\alpha$ will have descents at positions contained in $D$. Hence, by the first paragraph of this proof, $\kappa_{\alpha}\in \Pi_D$. It is well-known, and not hard to show, that
\begin{equation}
\label{eqn:Jan6abc}
[x^{\alpha}]\kappa_{\alpha}=1
\end{equation}
(this can be deduced from, \emph{e.g.}, Kohnert's rule \cite{Kohnert}).
Let $\prec$ be the reverse lexicographic order on monomials.
Among $(\lambda^{1},\ldots, \lambda^{k})\in {\sf Par}_D$ such that $c_{\lambda^{1},\lambda^{2},\ldots,\lambda^{k}}\neq 0$, pick the unique one such that $\alpha$ (as constructed above) is largest under $\prec$.
Now, $\alpha$ is the largest (monomial) exponent vector 
appearing in $g$ under $\prec$. This follows by an easy induction. The base case is that
that $\overline{\mu}$ is the $\prec$ largest exponent vector of $s_{\mu}$, which is well-known.

Hence in view of (\ref{eqn:Jan6abc}),
$g_1:=g-c_{\lambda^{1},\lambda^{2},\ldots,\lambda^{k}}\kappa_{\alpha}\in \Pi_D$
and the largest monomial appearing in $g_1$ is strictly smaller in $\prec$. Therefore
we may repeat this argument with $g_1$ to obtain $g_2$ and so on. As this process
eventually terminates with $g_r=0$. The result follows.
\end{proof}

\begin{example}
Let $n=4$ and $D=\{2\}$, then
\begin{align*}
g= &\  x_1 x_2^2 x_4+x_1^2 x_2 x_4+x_1 x_2^2 x_3+x_1^2 x_2 x_3+x_1^2 x_2^2 \in \Pi_D\\
= & \ s_{(2,1),(1,0)}+s_{(2,2),(0,0)}\\
= & \ \kappa_{1,2,0,1}+\kappa_{2,2,0,0}.
\end{align*}
Now, $(1,2,0,1), (2,2,0,0) \in {\sf Comp}_n(D)$, in agreement with Theorem~\ref{prop:ZBasis}.
\end{example}

Essentially the same argument for Proposition~\ref{prop:ZBasis}
establishes an analogous result for Schubert and Grothendieck polynomials. Split-symmetry of these polynomials was studied in connection
to \emph{degeneracy loci}, in \cite{BKTY,BKTY2}.

\subsection{Split-symmetry and multiplicity-free problems} 
Consider two disparate notions of \emph{multiplicity-freeness} that have been studied in algebraic combinatorics:
\begin{itemize}
\item[(MF1)] Suppose $f\in {\sf Sym}(n)$ and 
\[f=\sum_{\lambda\in {\sf Par}_n} c_{\lambda} s_{\lambda}.\]
Then $f$ is \emph{multiplicity-free} if $c_{\lambda}\in \{0,1\}$ for all $\lambda$.
J.~Stembridge~\cite{STEM} classified multiplicity-freeness
when $f=s_{\mu}s_{\nu}$. For more such classifications see, \emph{e.g.},~\cite{Bessenrodt, TY03, Gutschwager,
Bessenrodt.Willigenburg, BP14, Bessenrodt.Bowman}. 
\item[(MF2)]
Now let
\[f=\sum_{\alpha\in {\sf Comp}_n} c_{\alpha} x^{\alpha} \in {\sf Pol}_n.\]
$f$ is \emph{multiplicity-free} if $c_{\alpha}\in \{0,1\}$ for all
$\alpha$. In recent work of 
A.~Fink-K.~M\'esz\'aros-A.~St.~Dizier~\cite{FMSD}, multiplicity-free Schubert polynomials are
characterized.
\end{itemize}

We unify problems of type (MF1) and (MF2), as follows:

\begin{definition}[$D$-multiplicity-freeness]
\label{def:mfreeness}
\begin{equation}\label{eqn:June12abc}
f=\sum_{(\lambda^{1},\ldots,\lambda^{k}) \in 
{\sf Par}_D} c_{\lambda^{1},\ldots,\lambda^{k}} s_{\lambda^{1},\ldots,\lambda^{k}} \in
{\Pi}_D
\end{equation}
is \emph{$D$-multiplicity-free} if $c_{\lambda^{1},\ldots,\lambda^{k}}\in \{0,1\}$ for all
$(\lambda^{1},\ldots,\lambda^{k})\in {\sf Par}_D$.
\end{definition}

If $D=\emptyset$, Definition~\ref{def:mfreeness} is (MF1).
When $D=[n-1]$, notice ${\sf Par}_D={\sf Comp}_n$ and we recover (MF2). 

\begin{definition}[Composition patterns]
Let 
\[{\sf Comp} := \bigcup_{n=1}^{\infty}  {\sf Comp}_n.\]  
For $\alpha=(\alpha_1,\ldots,
\alpha_{\ell}), \beta=(\beta_1,\ldots, \beta_{k}) \in {\sf Comp}$, $\alpha$ \emph{contains the composition pattern} $\beta$ 
if there exists integers $j_1 < j_2 < \cdots < j_k$ that satisfy:
\begin{itemize}
\item $(\alpha_{j_{1}},\ldots,\alpha_{j_{k}})$ is order isomorphic to $\beta$ ($\alpha_{j_s} \leq \alpha_{j_t}$ if and only if $\beta_{s} \leq \beta_{t})$, 
\item $|\alpha_{j_s} - \alpha_{j_t}| \geq |\beta_{s} - \beta_{t}|$.
\end{itemize}
The first condition is the na\"ive notion of pattern containment, while the second allows for minimum relative differences. 
If $\alpha$ does not contain $\beta$, then $\alpha$ \emph{avoids} $\beta$. For $S \subset {\sf Comp}$, $\alpha$ avoids $S$ if $\alpha$ avoids all the compositions in $S$.
\end{definition}

\begin{example}
The composition $(3,\underline{1},4,\underline{2},\underline{2})$ contains $(0,1,1)$. It avoids $(0,2,2)$.
\end{example}

Define \[ {\sf KM} = \{ (0,1,2), (0,0,2,2), (0,0,2,1), (1,0,3,2), (1,0,2,2) \}. \] 
Let ${\overline {\sf KM}}_n$ be those $\alpha\in {\sf Comp}_n$ that avoid ${\sf KM}$.

\begin{theorem}
\label{thm:mfKey}
$\kappa_{\alpha}$ is $[n-1]$-multiplicity-free if and only if $\alpha \in {\overline {\sf KM}}_n$.
\end{theorem}

The proof is given in the companion paper \cite{Hodges.Yong}. The following problem asks for a complete
generalization of Theorem~\ref{thm:mfKey}:

\begin{problem}\label{prob:charD}
Fix $D\subseteq [n-1]$. Characterize $\alpha\in {\sf Comp}_n(D)$ such that $\kappa_{\alpha}$ is
$D$-multiplicity-free.
\end{problem}

C.~Ross and the second author \cite[Theorem~1.1]{Ross.Yong}
provide a (positive) combinatorial rule for computing the $D$-split expansion of $\kappa_{\alpha}$; this rule
is reproduced in Section~\ref{sec:bigrassmannianpf}.\footnote{Similarly, it would also be interesting to generalize
\cite{FMSD}. There is a formula of A.~Buch-A.~Kresch-H.~Tamvakis and the second author \cite{BKTY} for the split expansion of Schubert polynomials.} As we explain in the next subsection, this problem is of significance to the sphericality question.

In Example~\ref{exa:toricJune17}, we referred to the following compound result:

\begin{theorem}[cf.~\cite{Karupp} \cite{Tenner}] \label{thm:toric}
Let $w\in W={\mathfrak S}_n$. The following are equivalent:
\begin{itemize}
\item[(I)] $X_w \subset GL_n / B$ is a toric variety (with respect to the maximal torus $T$, \emph{i.e.,} $X_w$ is
$L_{\emptyset}$-spherical).
\item[(II)] $w=s_{r_1}\cdots s_{r_n}$ with $r_i \neq r_j$ for all $i \neq j$.
\item[(III)] $w$ avoids $321$ and $3412$.
\end{itemize}
\end{theorem}
\begin{proof}
The equivalence of (I) and (II) is in \cite{Karupp}, whereas the equivalence of (II) and (III) is proved in \cite{Tenner}.
\end{proof}

Using Theorem~\ref{thm:mfKey} we have an independent proof of (I) $\iff$ (III), that we omit for sake
of brevity. Since each of the $21$ bad patterns in ${\mathfrak S}_5$ from Example~\ref{ex:S123456} contains $321$
or $3412$, Theorem~\ref{thm:toric} gives evidence for Conjecture~\ref{conj:pattern}, because of Proposition~\ref{prop:geometricmono}.

\subsection{Sphericality and multiplicity-free key polynomials}
The key  polynomials have  a representation-theoretic interpretation \cite{Reiner.Shimozono, Ion, Mason}. Let 
$\mathfrak{X}(T)={\sf Hom}(T,\mathbb{C})$
be the \emph{character group} of $T$, with $\mathfrak{X}(T)^+$ the \emph{dominant integral weights}. For $\lambda \in \mathfrak{X}(T)$, $\lineb{\lambda}$ denotes the associated line bundle on $G/B$, as well as its restriction to Schubert subvarieties (cf. \cite[Chapter 2]{BL00}). Given $w \in W$ and $\lambda \in \mathfrak{X}(T)^+$ the \emph{Demazure module} is the dual of the space of sections of $\lineb{\lambda}$, $H^0(X_w,\lineb{\lambda})^{*}$~\cite{D74}. This space has a natural $B$-module structure induced by the action of $B$ on $X_w$. In \cite{Reiner.Shimozono}, the authors show that
\begin{equation}
\label{eqn:Bchar}
\text{$\kappa_{w \lambda}$ is the $B$-character of $H^0(X_w,\lineb{\lambda})^{*}$,}
\end{equation} 
where 
\begin{equation}
\label{eqn:wlambda}
w \lambda = (\lambda_{w^{-1}(1)},\ldots,\lambda_{w^{-1}(n)}).
\end{equation}
(A similar statement holds for all other finite types.)

The following summarizes the fundamental relationship between Levi spherical Schubert varieties,
Levi subgroup representation theory, Demazure modules, and split-symmetry:

\begin{theorem}
\label{thm:fundamentalRelationship}
Let $\lambda \in {\sf Par}_n$, and $w \in {\mathfrak S}_n$. Suppose $I\subseteq J(w)$ and $D=[n-1]-I$.
\begin{itemize}
\item[(I)]  $H^0(X_w,\lineb{\lambda})^{*}$ is an $L_I$-module with character $\kappa_{w \lambda}$. Hence $\kappa_{w\lambda}$ is a nonnegative integer combination of $D$-Schur polynomials in ${\Pi}_{D}$.
\item[(II)] $X_w$ is $L_I$-spherical if and only if  $\kappa_{w \lambda}$ is $D$-multiplicity-free
for all $\lambda \in {\sf Par}_n$.
\end{itemize}
\end{theorem}
\begin{proof}
Since $I \subseteq J(w)$, \eqref{eqn:stablizeractionparabolic} implies $L_I$ acts on $X_w$. 

(I) The action of $B$ on $H^0(X_w, \lineb{\lambda})^{*}$ is induced by the left multiplication action of $B$ on $X_w$~\cite{D74}. In the same way, the left multiplication action of $L_I$ on $X_w$ induces the $L_I$ action on $H^0(X_w, \lineb{\lambda})^{*}$. By \eqref{eqn:Bchar}, a diagonal matrix $x \in B$ acts on $H^0(X_w, \lineb{\lambda})^{*}$ with trace $\kappa_{w \lambda}$. The same diagonal matrix $x \in L_I$ acts identically on $H^0(X_w, \lineb{\lambda})^{*}$, and thus also has trace $\kappa_{w \lambda}$. Thus $\kappa_{w \lambda}$ is the character of an $L_I$-module. Since $L_I$ is reductive, and we work over a field of characteristic zero, character theory implies $\kappa_{w \lambda}$ may be written a nonnegative integer combination of characters of irreducible $L_I$-modules. That is, a nonnegative integer combination of $D$-Schurs in ${\Pi}_{D}$.

(II) There are numerous equivalent characterizations of spherical varieties found in the literature and collected in \cite[Theorem 2.1.2]{Perrin}. Of primary interest for us is the following: A quasi-projective, normal $R$-variety $Y$ is $R$-spherical for a reductive group $R$ if and only if the $R$-module $H^0(Y,\mathcal{L})$ is multiplicity-free for all $R$-linearized line bundles $\mathcal{L}$. 

All Schubert varieties are quasi-projective and normal~\cite{Jos.Normal}. The line bundles on $G/B$, when $G$ is of type $A$, are indexed by partitions in ${\sf Par}_n$. Every line bundle on $X_w$ is the restriction of a line bundle on $G/B$~\cite[Proposition 2.2.8]{Brion.Lectures}. Since $\mathcal{L}_{\lambda}$, for $\lambda \in {\sf Par}_n$, is $G$-linearized~\cite[\S 1.4]{Brion.Lectures}, its restriction to $X_w$, which we also denote by $\mathcal{L}_{\lambda}$, is $L_I$-linearized. Since $L_I$ is a product of general linear groups, $H^0(X_w, \lineb{\lambda})$ is a multiplicity-free $L_I$-module if and only if $H^0(X_w, \lineb{\lambda})^{*}$ is a multiplicity-free $L_I$-module. Thus, via the equivalent characterization of spherical varieties, we have that $X_w$ is $L_I$-spherical if and only if the $L_I$-module $H^0(X_w, \lineb{\lambda})^{*}$ is multiplicity-free for all $\lambda \in {\sf Par}_n$.  By (I), this holds if and only if $\kappa_{w \lambda}$ is $D$-multiplicity-free for all $\lambda \in {\sf Par}_n$.
\end{proof}

\begin{remark}
We similarly expect that Theorem \ref{thm:fundamentalRelationship} holds for $X_w$ in any $G/B$ and that $X_w$ is $L_I$-spherical if and only if all Demazure modules are multiplicity-free $L_I$-modules. We plan to explicate this in future work (with Y.~Gao). 
\end{remark}

\subsection{Consequences of Theorem~\ref{thm:fundamentalRelationship}} 
First, we illustrate how to reprove Proposition~\ref{prop:geometricmono}, in type $A_{n-1}$, but from symmetric function considerations:

\begin{corollary}[Geometric monotonicity (type $A_{n-1}$)]\label{cor:geometricmono}
Suppose $w\in {\mathfrak S}_n$ and $I'\subseteq I\subseteq J(w)$. If  $X_w$ is $L_{I'}$-spherical, then 
$X_w$ is $L_I$-spherical.
\end{corollary}
\begin{proof}
Suppose $X_w$ is not $L_I$-spherical. By Theorem~\ref{thm:fundamentalRelationship}(II), there exists
$\lambda\in {\sf Par}_n$ such that $\kappa_{w\lambda}$ is not $D$-multiplicity-free, where 
$D=[n-1]-I=\{d_1<d_2<\ldots<d_k\}$. That is,
\begin{equation}
\label{eqn:July7www}
\kappa_{w\lambda}=\sum_{(\lambda^1,\ldots,\lambda^k) \in {\sf Par}_D}c_{\lambda^1,\ldots,\lambda^k}s_{\lambda^1,\ldots,\lambda^k}
\end{equation}
and there exists $(\lambda^1,\ldots,\lambda^k)\in {\sf Par}_D$ such that
$c_{\lambda^1,\ldots,\lambda^k}>1$. 

By induction, we may assume $\#(I-I')=1$. Thus 
\[D':=[n-1]-I'=\{d_1<d_2<\ldots<d_f<d_f'<d_{f+1}<\ldots<d_k\}\supseteq D.\]
In general, let $\mu\in {\sf Par}_m$. Then it is standard (see, \emph{e.g.,} \cite[(7.66)]{ECII}) that
\begin{equation}
\label{eqn:July7yyy}
s_{\mu}(x_1,\ldots,x_m)=\sum_{\pi,\theta}C_{\pi,\theta}^{\mu}s_{\pi}(x_1,\ldots,x_a)s_{\theta}(x_{a+1},\ldots,x_m)
\end{equation}
where $C_{\pi,\theta}^{\mu}\geq 0$ is the \emph{Littlewood-Richardson coefficient}. Now apply (\ref{eqn:July7yyy}) to
each term of (\ref{eqn:July7www}): $\mu=\lambda^f$, $m=d_f-d_{f-1}$ and $a=d_{f'}-d_f$. Thereby, we obtain
a $D'$-Schur expansion of $\kappa_{w\lambda}$ in $\Pi_{D'}$ which also must have multiplicity. Now apply 
Theorem~\ref{thm:fundamentalRelationship}(II) once more.
\end{proof}

Second, towards Problem~\ref{prob:charD}, we offer:

\begin{theorem}
\label{cor:mfkeyweak}
Suppose $\alpha \in {\overline{\sf KM}}_n\cap {\sf Comp}_n(D)$. $\kappa_{\alpha}$ is $D$-multiplicity-free if either:
\begin{itemize}
\item[(I)] $\alpha\in {\sf Comp}_n$ has all parts distinct, that is, $\alpha_i\neq \alpha_j$ for $i\neq j$; or
\item[(II)] $\alpha$ also avoids $(0,0,1,1)$.
\end{itemize}
\end{theorem}
\begin{proof}
(I): Let $\lambda$ be the partition obtained by sorting the parts of $\alpha$ in decreasing order. Let $w\in {\mathfrak S}_n$
be such that $w\lambda=\alpha$ (this permutation is unique by the distinct parts hypothesis). We claim $w$ avoids
$321$ and $3412$. Suppose not. Observe that since $321$ and $3412$ are self-inverse, this means $w^{-1}$ contains
a $321$ or $3412$ pattern. In the former case, let $i<j<k$ be the indices of the $321$ pattern. Then 
$(\alpha_i,\alpha_j,\alpha_k)=(\lambda_{w^{-1}(i)},\lambda_{w^{-1}(j)}, \lambda_{w^{-1}(k)})$ and 
since $w^{-1}(i)>w^{-1}(j)>w^{-1}(k)$, we have $\lambda_{w^{-1}(i)}<\lambda_{w^{-1}(j)}< \lambda_{w^{-1}(k)}$
which means $\alpha_i<\alpha_j<\alpha_k$ is a $(0,1,2)$-pattern, a contradiction. Similarly, one argues that if
$w^{-1}$ contains a $3412$ pattern, then $\alpha$ contains $(1,0,3,2)$, another contradiction.

Hence $w$ avoids $321$ and $3412$. So, by Theorem~\ref{thm:toric}, $X_w$ is $L_{\emptyset}$-spherical.
Thus, by Theorem~\ref{thm:fundamentalRelationship}(II), $\kappa_{w\lambda}=\kappa_{\alpha}$ is multiplicity-free.
Now apply Corollary~\ref{cor:geometricmono} (or Proposition~\ref{prop:geometricmono}).

(II) Let $\lambda$ be as above. Since $\alpha$ might not have distinct parts, there is a choice of $w$ such that
$w\lambda=\alpha$. Choose $w$ such that if 
\begin{equation}
\label{eqn:July10ytt}
\alpha_i=\alpha_j \text{\ and $i<j \Rightarrow w^{-1}(i)<w^{-1}(j)$.}
\end{equation} 

We claim $w$ (equivalently $w^{-1}$) avoids $321$ and $3412$. Suppose not. Say $w^{-1}$ contains $321$ at positions
$i<j<k$. Then by (\ref{eqn:July10ytt}) this means $(\lambda_{w^{-1}(i)}<\lambda_{w^{-1}(j)}< \lambda_{w^{-1}(k)})$ and
hence $\alpha_i<\alpha_j<\alpha_k$ forms a $(0,1,2)$ pattern, a contradiction. Thus suppose $w^{-1}$ contains a
$3412$ pattern at $i<j<k<\ell$. By the same reasoning, we know 
$\alpha_i\geq \alpha_j, \alpha_j<\alpha_k\geq \alpha_{\ell}, \alpha_{\ell}>\alpha_i$. 

\noindent
\emph{Case 1:} ($\alpha_i=\alpha_j$) If $\alpha_k=\alpha_j+1$ then $\alpha_\ell=\alpha_k$ (otherwise we contradict
(\ref{eqn:July10ytt}). Then $\alpha$ contains $(0,0,1,1)$, a contradiction. Otherwise $\alpha_k\geq \alpha_j+2$,
and $\alpha$ contains $(0,0,2,2)$ or $(0,0,2,1)$.

\noindent
\emph{Case 2:} ($\alpha_i>\alpha_j$) Since $\alpha_k\geq \alpha_{\ell}>\alpha_i$, $\alpha$ contains $(1,0,3,2),(1,0,2,2)$,
a contradiction.

Hence $w^{-1}$ avoids $321$ and $3412$, and we conclude as in (I).
\end{proof}

Combining Theorem~\ref{cor:mfkeyweak} with the arguments of \cite[Section~3.1]{Hodges.Yong} gives a relatively short
proof of Theorem~\ref{thm:mfKey} under the additional hypothesis (I) or (II). However, there
is an obstruction to carrying out the argument to prove Theorem~\ref{thm:mfKey} completely. Consider $\alpha=(0,0,1,1)$. Indeed $\kappa_{\alpha}$ is $[n-1]$-multiplicity-free. Following the reasoning of the argument,
$\lambda=(1,1,0,0)$. The permutations $w\in {\mathfrak S}_4$ such that $w\lambda=\alpha$
are $3412,4312,3421,4321$, but each of these contains $321$ or $3412$. In \cite{Hodges.Yong},
we prove Theorem~\ref{thm:mfKey} using a different, purely combinatorial approach.

Third, we examine the following observation that is immediate from
Theorem~\ref{thm:fundamentalRelationship}(II):

\begin{corollary}
\label{Cor:May20abd}
Suppose $w\in {\mathfrak S}_n$ and $I\subseteq J(w)$.
Let $\lambda^{\text{staircase}}=(n,n-1,n-2,\ldots,3,2,1)$.
If $\kappa_{w\lambda^{\text{staircase}}}$ is not $D$-multiplicity-free then
$X_w$ is not $L_I$-spherical.
\end{corollary}

\begin{example}\label{exa:May20ddd}
Let $n=5$ and $w=24531$. Then $X_w\subset {GL}_5/{B}$. In Example~\ref{exa:24531predict}, we showed $w$ is not $J(w)$-spherical.
We now show this agrees with Conjecture~\ref{conj:main}. Let $I=J(w)=\{1,3\}$ and thus $D=\{2,4\}$. Since $w^{-1}=51423$, 
$w\lambda^{\text{staircase}}=w(5,4,3,2,1)=(1,5,2,4,3)$. 
Now, $\kappa_{w\lambda^{\text{staircase}}}\in {\Pi}_{D}$ and 
\begin{multline}\label{eqn:May20xyy}
\kappa_{1,5,2,4,3}=s_{(5,4),(2,1),(3)}+s_{(5,4),(3,2),(1)}+s_{(5,2),(3,2),(3)}
+2s_{(5,3),(3,2),(2)} +s_{(5,3),(2,2),(3)}\\+s_{(5,2),(3,3),(2)}+2s_{(5,2),(4,2),(2)}
+s_{(5,3),(3,3),(1)}+s_{(5,3),(4,1),(2)}+s_{(5,3),(3,1),(3)}\\+s_{(5,3),(4,2),(1)}+s_{(5,2),(4,3),(1)}
+s_{(5,2),(4,1),(3)}+s_{(5,4),(2,2),(2)}+s_{(5,4),(3,1),(2)}\\+s_{(5,1),(4,2),(3)}+s_{(5,1),(4,3),(2)}.
\end{multline}
By Corollary~\ref{Cor:May20abd}, the multiplicity in (\ref{eqn:May20xyy}) says that $X_w$ is not $L_{J(w)}$-spherical.

A theorem of V.~Lakshmibai-B.~Sandhya \cite{Lak} states that $X_w$ is smooth if and only if $w$ avoids the patterns $3412$
and $4231$. Hence $X_{24531}$ is smooth, but not spherical.
\end{example}

Theorem~\ref{thm:fundamentalRelationship} does not give an algorithm to prove $X_w$ is $I$-spherical, because it
demands that one check $\kappa_{w\lambda}$ is $D$-multiplicity-free for \emph{infinitely} many $\lambda$. A
complete solution to Problem~\ref{prob:charD} should give a characterization of when $X_w$ is $I$-spherical.
However, one can obtain an algorithm without solving that problem. The next claim asserts 
this infinite check can be reduced to a \emph{single} check. 

\begin{conjecture}
\label{conj:sphericalGeneric}
The converse of Corollary~\ref{Cor:May20abd} is true.
\end{conjecture}

Let us also state a weaker assertion:

\begin{conjecture}
\label{conj:weakerver}
If $X_w$ is not $L_I$-spherical, there exists $\lambda^{\text{distinct}}=(\lambda_1>\lambda_2>\ldots>\lambda_n)$
such that $\kappa_{w\lambda^{\text{distinct}}}$ is not $D$-multiplicity-free.
\end{conjecture}

\begin{conjecture}
\label{conj:upone}
Fix $D=\{d_1<d_2<\ldots<d_k\}$ and suppose $\alpha,\alpha^{\uparrow}\in {\sf Comp}_n(D)$ where 
$\alpha^{\uparrow}=(\alpha_1,\ldots,\alpha_{j-1},\alpha_j+1,\alpha_{j+1},\ldots,\alpha_n)$ for some $j$ such that
$\alpha_{j}+1\neq \alpha_i$ for all $i\neq j$. If
$\kappa_{\alpha}$ is not $D$-multiplicity-free, then $\kappa_{\alpha^{\uparrow}}$ is not $D$-multiplicity-free.
\end{conjecture}

\begin{lemma}
\label{lemma:Nov25}
Suppose $I\subseteq J(w)$ and $D=[n-1]-I$. Let $\lambda^{\text{distinct}}=(\lambda_1>\lambda_2>\ldots
>\lambda_n)$. Then $w\lambda^{\text{distinct}}\in {\sf Comp}_n(D)$. 
\end{lemma}
\begin{proof}
If $d\not\in D$ then $d\in I\subseteq J(w)$. Hence $w^{-1}(d)>w^{-1}(d+1)$ and
$\lambda_{w^{-1}(d)}^{\text{distinct}}<
\lambda_{w^{-1}(d+1)}^{\text{distinct}}$. So all descents of  $w\lambda^{\text{distinct}}$ must be in $D$, as desired.
\end{proof}

\begin{proposition}
\label{prop:allthesame}
Conjecture~\ref{conj:upone} $\Rightarrow$ Conjecture~\ref{conj:sphericalGeneric}.
\end{proposition}
\begin{proof}
Suppose $X_w$ is not $L_I$-spherical for some $I\subseteq J(w)$. 

First we show the weaker claim that Conjecture~\ref{conj:upone} $\Rightarrow$ Conjecture~\ref{conj:weakerver}:
By Theorem~\ref{thm:fundamentalRelationship}(II),
there exists $\lambda$ such that $w\lambda\in {\sf Comp}_n(D)$ and $\kappa_{w\lambda}$ is not $D$-multiplicity-free. If $\lambda^{(0)}:=\lambda$ has distinct parts, let $\lambda^{\text{distinct}}:=\lambda$. If not, consider the smallest $j_0$ such that $\lambda_{j_0}=\lambda_{j_0+1}$. Then define
\begin{equation}
\label{eqn:Nov30abcdef} 
\lambda^{(0,j)}=(\lambda_1+1,\lambda_2+1,\ldots,\lambda_j+1,\lambda_{j+1},\lambda_{j+2},\ldots,\lambda_{j_0},\lambda_{j_0+1},\ldots,\lambda_n), \text{ \ for $1\leq j\leq j_0$}.
\end{equation}
We let $\lambda^{(0,0)}:=\lambda^{(0)}$.
Since $\lambda^{(0,j)}$ and $\lambda^{(0,j-1)}$ only differ at position $j$ (by a single increment),
it is immediate from the
definitions (\ref{eqn:wlambda}) and (\ref{eqn:Nov30abcdef}) as well as the minimality of $j_0$ that the set of descent positions ${\sf Desc}(w\lambda^{(0,j)})$ of $w\lambda^{(0,j)}$ contains ${\sf Desc}(w\lambda^{(0,j-1)})$ for $1<j\leq j_0$. Now repeat this modification
with $\lambda^{(1)}:=\lambda^{(0,j_0)}$ replacing the role of $\lambda^{(0)}$. The minimal $j_1$ such that
$\lambda^{(1)}_{j_1}=\lambda^{(1)}_{j_1+1}$ satisfies $j_1>j_0$; we similarly construct new partitions
$\lambda^{(1,j)}$ where $1\leq j\leq j_1$. Hence after a finite number of iterations, we arrive at $\lambda^{\text{distinct}}:=\lambda^{(q)}:=\lambda^{(q-1,j_{q-1})}$ with distinct parts. Inductively,
\[{\sf Desc}(w\lambda^{(p,j)})\subseteq
{\sf Desc}(w\lambda^{(q)})\subseteq D,\]
where the rightmost containment is by Lemma~\ref{lemma:Nov25}. Hence, 
$w\lambda^{(p,j)}\in {\sf Comp}_n(D)$ for $0\leq p< q$ and $1\leq j\leq j_p$.
Conjecture~\ref{conj:upone} says that if we have $\alpha=w\lambda^{(p,j-1)}\in {\sf Comp}_n(D)$ and $\alpha^{\uparrow}=w\lambda^{(p,j)}\in {\sf Comp}_n(D)$ such that $\kappa_{\alpha}$ is not $D$-multiplicity-free, then $\kappa_{\alpha^{\uparrow}}$ is not $D$-multiplicity-free. Applying Conjecture~\ref{conj:sphericalGeneric}
repeatedly we see by induction that $\kappa_{w\lambda^{\text{distinct}}}$ is not $D$-multiplicity-free, as desired.

Conjecture~\ref{conj:upone} $\Rightarrow$ Conjecture~\ref{conj:sphericalGeneric}:
By the previous paragraph, assume there exists $\lambda^{[0]}=\lambda^{\text{distinct}}$
such that $\kappa_{w\lambda^{\text{distinct}}}$ is not $D$-multiplicity-free. If $\lambda^{[1]}:=(\lambda^{[0]}_1+1,\lambda^{[0]}_2, \ldots,\lambda^{[0]}_n)$, then by Conjecture~\ref{conj:upone},
$\kappa_{w\lambda^{[1]}}$ is not $D$-multiplicity-free. Iterating this argument, it follows that if 
$\lambda^{[1]}:=(\lambda^{[0]}_1+h,\lambda^{[0]}_2, \ldots,\lambda^{[0]}_n)$ 
for any $h\geq 1$, the same conclusion
holds. For the same reason, if $h>h'$ we can ensure 
$\lambda^{[2]}=(\lambda^{[0]}_1+h,\lambda^{[0]}_2+h',\lambda^{[0]}_3,\ldots, \lambda^{[0]}_n)$ has that $\kappa_{w\lambda^{[2]}}$ is not $D$-multiplicity-free. Continuing this line of
reasoning, we can conclude that there is $r\in {\mathbb N}$ such that 
$\overline\lambda:=\lambda^{\text{staircase}}+(r,r,\ldots,r)$ 
and $\kappa_{w\overline{\lambda}}$ is not $D$-multiplicity-free.

Now, either directly from the definition of key polynomials from Section~\ref{sec:keysxyz}, or, \emph{e.g.,} 
from Kohnert's rule \cite{Kohnert} we have:
\begin{equation}
\label{eqn:July6zzz}
\kappa_{w\overline\lambda}=\left(\prod_{i=1}^n x_i^r\right) \times \kappa_{w\lambda^{\text{staircase}}}.
\end{equation}
If $\mu\in {\sf Par}_d$ then it is easy to see from the definition of Schur polynomials that
\begin{equation}
\label{eqn:July6abc}
(y_1\ldots y_d)^r\times s_{\mu}(y_1,\ldots,y_d)=s_{r^d+\mu}(y_1,\ldots, y_d),
\end{equation}
where 
$r^d+\mu=(r+\mu_1,r+\mu_2,\ldots,r+\mu_d)$.

Combining (\ref{eqn:July6zzz}), (\ref{eqn:July6abc}) and the presumption that
$\kappa_{w\overline\lambda}$ is not $D$-multiplicity-free, we see that
$\kappa_{w\lambda^\text{staircase}}$ is not $D$-multiplicity-free, as desired.
\end{proof}

In turn, it seems plausible to prove Conjecture~\ref{conj:upone} using \cite[Theorem~1.1]{Ross.Yong}. We hope to
address this in a sequel. For now, we offer the following evidence for its correctness.

\begin{proposition}
Conjecture~\ref{conj:upone} holds for $D=[n-1]$. 
\end{proposition}
\begin{proof}
This follows from Theorem~\ref{thm:mfKey} in this fashion: Suppose $\kappa_{\alpha}$ is not $[n-1]$-multiplicity-free
since $(\alpha_a,\alpha_b,\alpha_c,\alpha_d)$ is the pattern $(1,0,3,2)$. If $j\not\in \{a,b,c,d\}$ then $\alpha^{\uparrow}$
still contains $(1,0,3,2)$. 
If $j=a$ then (by the hypothesis of Conjecture~\ref{conj:upone}) 
$\alpha_a+1\neq \alpha_d$ hence $\alpha^{\uparrow}$ contains $(1,0,3,2)$ at the same positions. The same conclusion
holds if $j=b,c,d$. Hence by Theorem~\ref{thm:mfKey}, $\kappa_{\alpha^{\uparrow}}$ is not $[n-1]$-multiplicity-free. The other cases are left to the reader.
\end{proof}

\section{Proof of the bigrassmannian theorem}\label{sec:bigrassmannianpf}

Using the preparation in Sections~\ref{sec:2} and~\ref{sec:4}, we are now ready to prove Theorem~\ref{thm:bigrassmannian}.

First, we prove that three classes (\ref{eqn:3cases}) of bigrassmannian $w\in {\mathfrak S}_n$ are $J(w)$-spherical. Suppose ${\text{code}}(w)=(0^f,a,0^g)$. Then the canonical reduced word (see Example~\ref{exa:canonicalred}) is 
\[R^{\text{canonical}}(w)=s_{f+a} s_{f+a-1}\cdots s_{f+2}s_{f+1}.\] 
By Lemma~\ref{lemma:bigrassfact}, $J(w)=\{f+a\}$. Since $R^{\text{canonical}}(w)$ uses distinct generators, it is the $J(w)$-witness, as desired.
Similarly, one argues the case that ${\text{code}}(w)=(0^f,1^b,0^g)$. Finally, suppose ${\text{code}}(w)=(0^f,2^2,0^g)$.
In this case, 
\[R^{\text{canonical}}(w)=s_{f+2}s_{f+1}s_{f+3}s_{f+2}.\] 
Since (by Lemma~\ref{lemma:bigrassfact}) $J(w)=\{f+2\}$ we see that
$R^{\text{canonical}}(w)$ is again a $J(w)$-witness, as desired.

Conversely, suppose that $w\in {\mathfrak S}_n$ is bigrassmannian, but not one of the three cases (\ref{eqn:3cases}). Thus, 
$D(w)$ either has at least three columns, or at least three rows. Assume it is the former case (the argument for the
latter case is similar). Look at the canonical filling of $D(w)$. In the northwest $2\times 3$ subrectangle, the filling,
read right to left and top down is
\begin{equation}
\label{eqn:June15aaa}
s_{f+3}s_{f+2}s_{f+1}s_{f+4}s_{f+3}s_{f+2}.
\end{equation}
Let $u$ be the associated permutation
and $R^u\in {\text{Red}}(u)$ be the expression (\ref{eqn:June15aaa}). $R^u$ is a subexpression of $R^{\text{canonical}}(w)$.
Hence by Theorem~\ref{thm:subword}, $u\leq v$. By inspection, any $R'\in {\text{Red}}(u)$ has at least two
$s_{f+2}$'s. By assumption, $J(w)=\{d\}$ where $d\geq f+3$ (here we are again using Lemma~\ref{lemma:bigrassfact}). So every $R'$ fails
(S.1') (with respect to $J(w)$). Thus by Proposition~\ref{prop:subtrick}, $w$ is not $J(w)$-spherical.

Next, we show that for $w\in {\mathfrak S}_n$ satisfying (\ref{eqn:3cases}), $X_w$ is $L_{J(w)}$-spherical. First suppose
\[{\text{code}}(w)\in \{(0^f,a,0^g),(0^f,1^b,0^g)\}.\] 
The above analysis shows that $R^{\text{canonical}}(w)$ satisfies
Theorem~\ref{thm:toric}(II). Hence $X_w$ is a toric variety (by the equivalence (I)$\iff$(II) of said theorem). By
Corollary~\ref{cor:geometricmono} (or Proposition~\ref{prop:geometricmono}), $X_w$ is $L_{J(w)}$-spherical. Lastly, suppose 
\[{\text{code}}(w)=(0^f,2^2,0^g).\] 
First, assume $f=0$. Hence in this case
the permutation is $w'=s_2 s_1 s_3 s_2\in {\mathfrak S}_4$. Now $w'=3412$ in one-line notation, and $J(w')={\{2\}}$.

\begin{claim}
\label{claim:3412spherical}
$X_{3412B} \subset GL_4/B$ is $L_{\{2\}}$-spherical. 
\end{claim}
\noindent
\emph{Proof of Claim~\ref{claim:3412spherical}:}
Fix $B^{ss} := SL_4 \cap B$, and $T^{ss} := SL_4 \cap T$ as our choice of Borel subgroup and maximal Torus in $SL_4$. For $I \subseteq 2^{[3]}$, let $L^{ss}_I \leq P^{ss}_I$ denote the associated Levi and parabolic subgroups in $SL_4$. We prove $X_{3412B^{ss}} \subset SL_4/B^{ss}$ is $L^{ss}_{\{2\}}$-spherical. Since $SL_n/B^{ss} \cong GL_n/B$ as $SL_n$-varieties, this induces an $L^{ss}_{\{2\}}$-equivariant isomorphism between $X_{3412B^{ss}}$ and $X_{3412B}$. Thus if $X_{3412B^{ss}}$ is $L^{ss}_{\{2\}}$-spherical, then $X_{3412B}$ is $L^{ss}_{\{2\}}$-spherical. Since $L^{ss}_{\{2\}} \leq L_{\{2\}}$, this in turn implies $X_{3412B}$ is $L_{\{2\}}$-spherical.

The canonical projection $\pi : SL_4/B^{ss} \rightarrow SL_4 / P^{ss}_{\{1,3\}}$ induces a birational morphism \[X_{3412B^{ss}} \rightarrow X_{3412P^{ss}_{\{1,3\}}} \cong SL_4 / P^{ss}_{\{1,3\}}.\] Since $\pi$ is $SL_4$-equivariant, this birational morphism is $L^{ss}_{\{2\}}$-equivariant. Thus $X_{3412B^{ss}}$ is $L^{ss}_{\{2\}}$-spherical if and only if $SL_4 / P^{ss}_{\{1,3\}}$ is $L^{ss}_{\{2\}}$-spherical. As noted in the proof of Theorem~\ref{prop:agreewithMWZ}, the latter holds if and only if $SL_4 / P^{ss}_{\{2\}} \times SL_4 / P^{ss}_{\{1,3\}}$ is spherical for the diagonal $SL_4$ action. Finally, by~\cite{Stem:Weyl}[Corollary 1.3.A(ii)] this diagonal action is spherical. \qed

For general $f$, since $w=s_{f+2} s_{f+1} s_{f+3} s_{f+2}$, 
in fact 
$w=\phi(w')$ where $\phi$ is the Dynkin diagram embedding of $\dynkin[labels={1',2',3'}, edge length=0.5cm]A3$
into $\dynkin[labels={1,2,n-1,n}, edge length=0.5cm]A{}$ that sends 
$1'\mapsto f+1, 2'\mapsto f+2, 3'\mapsto f+3$.
This induces a map of the Weyl groups that sends $w'$ to $w$. Now Claim~\ref{claim:3412spherical} and Proposition~\ref{prop:dynkingeo} imply that $X_w$ is $L_{J(w)}$-spherical.

It remains to show that if $w\in {\mathfrak S}_n$ does not satisfy (\ref{eqn:3cases}), then $X_w$ is not $L_{J(w)}$-spherical. Now,
$D(w)$ either contains a $2\times 3$ rectangle or a $3\times 2$ rectangle. Let us assume we are in the former 
case (the other case is similar, and left to the reader).

\begin{Claim}
\label{claim:June4abc}
If $D=\{1,2,3,\ldots,a-1,a+1,a+2\}$, and $a\geq 3$, then $\kappa_{0^a,2,1}$ is not $D$-multiplicity-free.
$s_{\emptyset^{a-3},(1),(1),(1,0),\emptyset,\ldots,\emptyset}$ appears in the expansion (\ref{eqn:June12abc}) of $\kappa_{0^a,2,1}$, with multiplicity (at least) $2$. 
\end{Claim}
\noindent
\emph{Proof of Claim~\ref{claim:June4abc}:} We recall \cite[Theorem~1.1]{Ross.Yong} which gives a nonnegative combinatorial rule to compute the expansion (\ref{eqn:June12abc})
of $f=\kappa_{\alpha}$ for any $\alpha\in {\sf Comp}_n(D)$. Let $w[\alpha]$ be the unique
permutation in ${\mathfrak S}_{\infty}$ such that $\text{code}(w[\alpha])=\alpha$ (ignoring any trailing $0$'s).
That such a permutation exists and is unique follows from, \emph{e.g.}, \cite[Proposition~2.1.2]{Manivel}.

 We now construct a tableau $T[\alpha]$. Given $w^{(1)}=w[\alpha]$, let $i_1$ be the position of
the last descent of $w^{(1)}$, and let $i_2$ be the location of the rightmost descent left of $i_1$ in 
$w^{(1)}s_{i_1}$ (so $i_2<i_1$). Repeat, defining 
$i_{j}$ to be the position of the rightmost descent to the left of $i_{j-1}$ in $w^{(1)}s_{i_1}s_{i_2}\cdots s_{i_{j-1}}$.
Suppose no descent appears left of $i_j$ in $w^{(1)}s_{i_1}s_{i_2}\cdots s_{i_{j}}$. In that case, stop, and, we define the first column
of $T[\alpha]$ to be filled by $i_1>i_2>\ldots>i_j$ (from bottom to top). Now let $w^{(2)}=w^{(1)}s_{i_1}s_{i_2}\cdots s_{i_{j}}$
and similarly we determine the entries of the second column. We repeat until we arrive at $k$ such that $w^{(k)}=id$.

An \emph{increasing tableau} $T$ of shape $\lambda$ is a filling of the Young diagram $\lambda$ with positive integers that is strictly increasing, left to right along rows, and top to bottom along columns. Let $\text{row}(T)$ be the right to left, top to bottom row reading word of $T$. Also let $\text{min}(T)$ be the
value of the minimum entry of $T$.

Given ${\bf a}=(a_1,a_2,\ldots)$ such that $s_{a_1}s_{a_2}\cdots$ is a reduced expression (for some
permutation), we will let $\tt{EGLS}({\bf a})$ be the \emph{Edelman-Greene column insertion tableau}; we
refer to \cite[Section~2.1]{Ross.Yong} for a summary of this well-known concept from algebraic combinatorics. 
Below, we will mildly abuse notation and refer to ${\bf a}$ and $s_{a_1}s_{a_2}\cdots$ interchangeably.

\begin{theorem}[Theorem~1.1 of \cite{Ross.Yong}]\label{RYthm}
Let $\alpha\in {\sf Comp}_n(D)$ and $f=\kappa_{\alpha}$. The coefficient $c_{\lambda^1,\ldots,\lambda^k}$ in the
expansion (\ref{eqn:June12abc}) counts the number of sequences of increasing tableaux $(T_1,\ldots,T_k)$ such that
\begin{itemize}
\item[(a)] $T_i$ is of shape $\lambda^i$
\item[(b)] $\text{min}(T_1)>0,\text{min}(T_2)>d_1,\ldots, \text{min}(T_k)>d_{k-1}$; 
\item[(c)] $\text{row}(T_1)\cdot \text{row}(T_2)\cdots \text{row}(T_k)\in \text{Red}(w[\alpha])$; and
\item[(d)] ${\tt{EGLS}}(\text{row}(T_1)\cdot \text{row}(T_2)\cdots \text{row}(T_k))=T[\alpha]$.
\end{itemize}
\end{theorem}

In our particular case, $\alpha=(0^a,2,1)$. Hence,
\[w[\alpha]=  12\cdots a \ a+3 \ a+2 \ a+1 \text{\ \ (one line notation)}=  s_{a+1}s_{a+2}s_{a+1} \equiv  s_{a+2}s_{a+1}s_{a+2}.\]
Then the two tableau sequences are 
\[(\emptyset^{a-3},\framebox{$a+1$}, \framebox{$a+2$}, \framebox{$a+1$},\emptyset,\ldots,
\emptyset) \text{\ \ and $(\emptyset^{a-3},\framebox{$a+2$}, \framebox{$a+1$}, \framebox{$a+2$},\emptyset,\ldots,\emptyset)$.}\]
Here  $T[\alpha]=
\ytableausetup{boxsize=1.5em}
\begin{ytableau}
\scalebox{0.7}{$a+1$} & \scalebox{0.7}{$a+2$} \\
\scalebox{0.7}{$a+2$} 
\end{ytableau}$. It is straightforward to check the conditions of Theorem~\ref{RYthm} are satisfied. In particular,
condition (d) is requiring that the Edelman-Greene column insertions of $a+1 \ a+2 \ a+1$ and $a+2 \ a+1 \ a+2$ both give
$T[\alpha]$; this is true.  (In fact these are the only valid tableau sequences for the datum, although we do not need this.)
\qed

\begin{Claim}
\label{claim:June4bbb}
Let $D'=\{1,2,3,\ldots,f,f+1,f+2,\ldots,f+(a-1),f+(a+1),f+(a+2)\}$ and $\alpha=(3^f,0^a,2,1)$
Then $\kappa_{\alpha}$ is not $D'$-multiplicity-free.
\end{Claim}
\noindent
\emph{Proof of Claim~\ref{claim:June4bbb}:}  Straightforwardly from \emph{Kohnert's rule} \cite{Kohnert},
\begin{equation}
\label{eqn:Junexxx}
\kappa_{\alpha}=\prod_{i=1}^f x_i^3 \times \kappa_{0^a,2,1}(x_{f+1},x_{f+2},\ldots,x_{f+(a+2)}).
\end{equation}
Suppose $c_{\lambda^1,\lambda^2,\lambda^3,\lambda^4,\ldots}$ is the coefficient of $s_{\lambda^1,\lambda^2,\lambda^3,
\lambda^4,\ldots}$ in the $D$-expansion (\ref{eqn:June12abc}) of $\kappa_{0^a,2,1}$. 
Let $c_{(3),(3),\ldots,(3),\lambda^1,\lambda^2,\lambda^3,\lambda^4,\ldots}$ be the $D'$-split-expansion of $\kappa_{\alpha}$
(here there are $f$-many $(3)$'s).
Then (\ref{eqn:Junexxx}) implies
\[c_{(3),(3),\ldots,(3),\lambda^1,\lambda^2,\lambda^3,\lambda^4,\ldots}=c_{\lambda^1,\lambda^2,\lambda^3,\lambda^4,\ldots}.\] Now apply Claim~\ref{claim:June4abc}.\qed

Since $\text{code}(w)=(0^f,a^b,0^g)$ where $a\geq 3$ and $b\geq 2$, 
\[w^{-1}=1 \ \ 2  \ \ 3 \cdots f \ \ f+b+1 \ \ f+b+2 \ \ \cdots \ \ f+b+a \ \ f+1\  \ f+2 \ \ \cdots\  \ f+b\ \  \cdots,\]
where the rightmost ``$\cdots$'' contains the remaining numbers from $[n]$ listed in increasing order.
Let
\[\lambda = \underbrace{3,3,\ldots,3}_{\text{$f$-many}},2,1,\underbrace{0,0,\ldots,0}_{\text{$(n-f-2)$-many}}.\]
Thus 
\[w\lambda:=(\lambda_{w^{-1}(1)},\ldots,\lambda_{w^{-1}(n)})=(3^f,0^a,2,1,0^{n-f-a-2}).\]
Set 
$D''=D'\cup\{f+(a+3),f+(a+4),f+(a+5),\ldots\}$. 
Hence it follows from Claim~\ref{claim:June4bbb} that $\kappa_{w\lambda}$ is not $D''$-multiplicity-free. 
By Lemma~\ref{lemma:bigrassfact}, $J(w)=\{f+a\}$, and hence $[n-1]-J(w)=D''$; therefore, $X_w$ is not $L_{J(w)}$-spherical, by Theorem~\ref{thm:fundamentalRelationship}(II).\qed


\section*{Acknowledgements}
We thank Mahir Can, Laura Escobar, Martha Precup, Edward Richmond, and John Shareshian for helpful discussions. We thank David Brewster, Jiasheng Hu, and Husnain Raza for writing useful computer code (in the
NSF RTG funded ICLUE program). We are grateful to the anonymous referee for their detailed comments which significantly improved the final presentation.
We used the Maple packages ACE and Coxeter/Weyl in our investigations.
AY was partially supported by a Simons Collaboration Grant, and an NSF RTG grant.
RH was partially supported by an AMS-Simons Travel Grant.

\end{document}